\documentclass[12pt]{amsart}
\usepackage{amssymb, amscd, tikz, latexsym, dsfont,amsmath,amsthm}

\newcommand{\field}[1]{\mathbb{#1}}
\newcommand{\CC}{\field{C}}

\newcommand{\NN}{\field{N}}
\newcommand{\RR}{\field{R}}
\newcommand{\TT}{\field{T}}
\newcommand{\ZZ}{\field{Z}}
\newcommand{\QQ}{\field{Q}}

\newcommand{\Aa}{\mathcal A}

\newcommand{\Ff}{\mathcal F}

\newcommand{\Kk}{\mathcal K}
\newcommand{\Ll}{\mathcal L}

\newcommand{\Qq}{\mathcal Q}
\newcommand{\Tt}{\mathcal T}
\newcommand\3[1]{{\mathds #1}}

\newcommand{\NT}{\operatorname{\mathcal{NT}}}

\newcommand{\Aut}{\operatorname{Aut}}

\newcommand{\id}{\operatorname{id}}
\newcommand{\Ind}{\operatorname{Ind}}

\newcommand{\clsp}{\operatorname{\overline{span\!}\,\,}}

\theoremstyle{plain}
\newtheorem{theorem}{Theorem}[section]
\newtheorem*{theorem*}{Theorem}
\newtheorem*{prop*}{Proposition}
\newtheorem{cor}[theorem]{Corollary}
\newtheorem{lemma}[theorem]{Lemma}
\newtheorem{prop}[theorem]{Proposition}

\newtheorem{rmk}[theorem]{Remark}

\newtheorem{example}[theorem]{Example}

\theoremstyle{definition}
\newtheorem{dfn}[theorem]{Definition}

\numberwithin{equation}{section}

\begin{document}

\title[KMS states on Nica-Toeplitz algebras of product systems]{KMS states on Nica-Toeplitz algebras
\\ of product systems}

\author[J. H. Hong]{Jeong Hee Hong$^{\dag}$}
\address{Department of Data Information, Korea Maritime University,
Busan 606--791, South Korea}
\email{hongjh@hhu.ac.kr}

\author[N. S. Larsen]{Nadia S. Larsen$^{\ddag\flat}$}
\address{Department of Mathematics, University of Oslo, PO Box 1053 Blindern,
N--0316 Oslo, Norway}
\email{nadiasl@math.uio.no}

\author[W. Szyma{\'n}ski]{Wojciech Szyma{\'n}ski$^{\sharp\flat}$}
\address{Department of Mathematics and Computer Science, University of Southern Denmark,
Campusvej 55, DK--5230 Odense M, Denmark}
\email{szymanski@imada.sdu.dk}

\thanks{\hspace{-4mm}{\dag}  J. H. Hong was supported by Basic Science Research
Program through the National Research
Foundation of Korea (NRF) funded by the Ministry of Education, Science and Technology (2010--0022884). \\
{\ddag} N. S. Larsen was supported by the Research Council of Norway. She acknowledges the inspiring
environment of the Banff International  Research Station during a stay in June 2012.\\
{$\sharp$} W. Szyma{\'n}ski was partially supported by the FNU Forskningsprojekt
`Structure and Symmetry' (2010--2012), the FNU Rammebevilling `Operator algebras and
applications' (2009--2011), and a travel grant from Danske Universiteter and the Japanese Society
for the Promotion of Science. \\
{$\flat$} N. S. Larsen and W. Szyma{\'n}ski were also supported by the NordForsk research network
"Operator Algebra and Dynamics" (grant \#11580).}

\begin{abstract}
We investigate  KMS states of Fowler's Nica-Toeplitz algebra $\mathcal{NT}(X)$ associated to
a compactly aligned product system $X$ over a semigroup $P$ of Hilbert bimodules. This analysis relies on restrictions
of these states to the core algebra which satisfy appropriate scaling conditions. The concept of product system of
finite type is introduced. If $(G, P)$ is a lattice ordered group
and $X$ is a product system of finite type over $P$ satisfying certain coherence properties, we construct
KMS$_\beta$ states of $\NT(X)$ associated to a scalar dynamics from traces on the coefficient algebra
of the product system. Our results were motivated by, and generalize some of the results
of Laca and Raeburn obtained for the Toeplitz algebra of the affine semigroup over the natural numbers.
\end{abstract}

\date{February 20, 2012. Revised July 17, 2012.}
\vskip 1cm

\maketitle

\section{Introduction}\label{section:Intro}
KMS$_\beta$ states for a quasi-free dynamics on the Toeplitz algebra associated to a right Hilbert
bimodule over a $C^*$-algebra have been constructed in many contexts by different authors.
 A unified approach that moreover greatly generalized earlier specific constructions was obtained
by Laca and Neshveyev in \cite{Lac-Nes}.

Recently $C^*$-algebras associated with rings and exhibiting an interesting structure of KMS states
have been discovered. In \cite{Cun1}, Cuntz associated $C^*$-algebras to the affine semigroup over
the natural numbers and to the ring of integers, and proved in both cases existence of a single
KMS$_\beta$ state at (inverse) temperature $\beta=1$ for a  natural dynamics.  Cuntz's $C^*$-algebra
$\mathcal{Q}_{\NN}$ of the affine semigroup over the natural numbers is purely infinite and simple.
It arises in a way reminiscent of a boundary construction, and this prompted Laca and Raeburn to find a Toeplitz
algebra $\Tt(\NN\rtimes\NN^\times)$ of the affine semigroup over the natural numbers with a much richer
structure of KMS states for a natural dynamics, \cite{Lac-Rae2}. They proved that
 $(\QQ\rtimes \QQ^*_+, \NN\rtimes\NN^\times)$ is a quasi-lattice ordered group in the sense
of Nica \cite{Ni}, and using results on the associated Nica spectrum established
that indeed $\mathcal{Q}_{\NN}$ is a boundary quotient of $\Tt(\NN\rtimes\NN^\times)$.
 In \cite{BanHLR}, Brownlowe, an Huef, Laca  and Raeburn analyze KMS states of the intermediary subquotients
of $\Tt(\NN\rtimes\NN^\times)$ with help of the techniques developed originally in \cite{Lac-Rae2}.
As recent closely related results we would like to mention
the analysis of KMS states of $\Tt(\NN\rtimes\NN^\times)$ based on a crossed product approach,
carried out by Laca and Neshveyev in \cite{Lac-Nes2}, and the classification of KMS states for Toeplitz-like
$C^*$-algebras associated to the ring
of integers in a number field, worked out by Cuntz, Deninger and Laca in \cite{CDL}.

Our goal in the present paper is to initiate the study of KMS states of the Nica-Toeplitz algebra of a product
system over a semigroup of Hilbert bimodules.  Our motivation  comes from the fact that  $\mathcal{Q}_{\NN}$
and several of its extensions  can be modeled as Cuntz-Pimsner and respectively Toeplitz-type
$C^*$-algebras associated to product systems over the semigroup $\NN^\times$ of Hilbert bimodules.
Indeed, it is shown in \cite{BanHLR}  that both $\Tt(\NN\rtimes\NN^\times)$ and its natural subquotients $\Tt_{add}$ and $\Tt_{mult}$  are the Nica-Toeplitz algebras of product systems over $\NN^\times$.
Both $\Tt_{add}$ and  $\mathcal{Q}_{\NN}$ were realized  by different methods
as algebras associated to product systems by the present authors in \cite{HLS}  and  by
Yamashita in \cite{Yam}, respectively.  Since the product system structure of these algebras
arises from very natural endomorphisms and corresponding transfer operators, it is worthwhile to
investigate the general problem of whether KMS states can be described systematically for $C^*$-algebras
associated to product systems over semigroups of Hilbert bimodules.

Given a cancellative semigroup $P$ with identity, a product system $X$ over $P$ of Hilbert bimodules
over a $C^*$-algebra $A$ is a family of right Hilbert $A$--$A$-bimodules $X_s$ for $s\in P$, which forms
a semigroup compatible with the Hilbert bimodule structure of the $X_s$'s. The Toeplitz algebra
$\Tt(X)$ is universal for Toeplitz representations of the product system $X$ which  respect the  semigroup multiplication,
 see  \cite{Fow}. When $(G, P)$ is a quasi-lattice ordered group and $X$ is a compactly aligned
product system, Fowler argued that a quotient of $\Tt(X)$  which encodes the Nica-covariant representations
of $X$ is the appropriate object of study. Following \cite{BanHLR}, we denote this quotient  $\NT(X)$ and
refer to it as the \emph{Nica-Toeplitz algebra of $X$}.

 In the case of a single right Hilbert $A$--$A$-bimodule $X$, Laca and Neshveyev \cite{Lac-Nes}
constructed KMS$_\beta$ states of $\Tt(X)$ for  $\beta\in (0, \infty)$ from certain traces of the coefficient
algebra $A$ by means of state extensions  to the fixed point algebra associated to the canonical
gauge action which satisfy a scaling-type condition. Composition with  the canonical conditional
expectation then gives rise to a state of the Toeplitz algebra which fulfills the KMS$_\beta$ condition.
 A generalization of the construction in \cite{Lac-Nes} was obtained by Kajiwara and Watatani  in
the case where the left action of the bimodule is not injective, \cite{Ka-Wa}.

 The algebra $\NT(X)$ carries a coaction of $G$, and admits a conditional expectation onto its fixed-point
algebra, or \emph{core},  $\Ff$. The basic strategy is to identify those states of $\Ff$ which extend
to KMS states of $\NT(X)$. For ground states, it is possible to give a very general
necessary and sufficient condition on states of $\Ff$ which extend to ground states of $\NT(X)$.
However, for $\beta\in(0, \infty)$, only a necessary condition on  KMS$_\beta$ states may be obtained
in the greatest generality. In either case, it is  desirable to reduce the problem further and characterize
KMS states in terms of states (or tracial states) of $A$.

\smallskip
We now briefly summarize the content of the present paper.
Taking into account our motivating examples (notably $\Tt(\NN\rtimes\NN^\times)$ and $\Tt_{add}$),
for large part of our analysis we  assume that $(G, P)$ is a lattice ordered group and that each bimodule
$X_s$ for $s\in P$ has a finite orthonormal basis (as a right Hilbert $A$-module) compatible with the
semigroup multiplication; we call such $X$ a \emph{product system of finite type}, see
section~\ref{section_finite-type}. In this setting, we identify a scaling condition on a tracial state $\phi$ of $\Ff$
which is sufficient for the composition of $\phi$ with the conditional expectation onto $\Ff$ to produce a KMS$_\beta$
state of $\NT(X)$, see Theorem~\ref{thm:Laca-12-KMSbeta}. Since this condition involves scaling
by isometries in $\NT(X)$  (in the case at hand these arise from the elements in the orthonormal bases for
the $X_s$), this result is similar in spirit to \cite[Theorem 12]{Lac}.

The next question we address is how to decide when a state of $\Ff$ is tracial and satisfies the scaling
condition. Under some fairly natural hypotheses on the number of elements in the orthonormal bases,
we show in Theorem~\ref{thm:from-trace-A-to-trace-F} that a state of $\Ff$ which restricts to a tracial
state on $A$ and satisfies the scaling condition on elements of  $A$ is a trace of $\Ff$.
Under a further condition of convergence  of a certain series in an interval $(\beta_c,\infty)$ for some critical value
$\beta_c>0$, we prove in Theorem~\ref{thm:KMsbeta-from-tracetau}  that for $\beta>\beta_c$ every
tracial state of $A$ gives rise to a state of $\Ff$ that satisfies the hypotheses of
 Theorem~\ref{thm:from-trace-A-to-trace-F}, and hence gives rise to a KMS$_\beta$ state of $\NT(X)$.
Ground states are much easier to analyze, see Proposition~\ref{prop:ground-states}.

In section~\ref{section:core}, we first establish some properties of the core $\Ff$, after which we
show that, under mild hypotheses on $P$, every KMS$_\beta$ state above a certain critical value
$\beta_0$ satisfies the reconstruction formula found by Laca and Raeburn in \cite[\S 10]{Lac-Rae2}.
Based on this we identify conditions on $X$ that guarantee that the parametrization of KMS states
in terms of traces on $A$ in Theorem~\ref{thm:KMsbeta-from-tracetau} is
 surjective (Corollary~\ref{cor:surjectivity-of-parametrization}) and injective
(Theorem~\ref{thm:inj-of-parametrization}). We compare our general results with the similar results obtained
for a specific system in  \cite{Lac-Rae2}, see Remark \ref{comparisonwithLR}.
A larger class of examples from the forthcoming paper \cite{Bro-Lar} by Brownlowe and the second author is discussed in
Example \ref{brla}.

Our methods of construction of KMS states reflect the properties of elements in the orthonormal bases
for the bimodules, so we think they are very natural in this setup. They also explain the origin of some of the
computations and resulting formulas from \cite{Lac-Rae2} upon viewing the $C^*$-algebra
$\Tt(\NN\rtimes\NN^\times)$ as a $\NT(X)$, cf. \cite{BanHLR}.

\smallskip
{\bf Acknowledgement.}
We would like to thank Nicolai Stammeier for useful comments on an earlier version of this paper.

\section{Preliminaries}\label{section:Prelim}


\subsection{Product systems of Hilbert bimodules}\label{Prelim-1}

Let $A$ be a $C^*$-algebra and $X$ be a complex vector space with a right action of $A$.
Suppose that there is an $A$-valued inner product
$\langle \cdot , \cdot \rangle _A$ on $X$ which is conjugate linear in the first variable
and linear in the second variable, and satisfies
\begin{enumerate}
\item $\langle\xi, \eta\rangle _A =\langle\eta , \xi\rangle_A^*$,
\item $\langle\xi, \eta\cdot a\rangle_A=\langle\xi, \eta\rangle_A\, a$,
\item $\langle\xi, \xi\rangle_A \geq 0 $ and $\langle\xi, \xi\rangle_A=0 \; \Longleftrightarrow\; \xi=0$,
\end{enumerate}
for $\xi,\eta\in X$ and $a\in A$. Then $X$ becomes a right Hilbert $A$-module
when it is complete with respect to the norm given by
$\|\xi\|:=\|\langle\xi,\xi\rangle _A\|^{\frac{1}{2}}$ for $\xi\in X$.

A linear map $T:X\rightarrow X $ is said to be adjointable if there is a map $T^*:X\rightarrow X$
such that $\langle T\xi,\zeta\rangle_A=\langle\xi,T^*\zeta\rangle_A$ for all $\xi,\eta\in X$.
The set $\Ll(X)$ of all adjointable operators on $X$ endowed with the operator norm
is a $C^*$-algebra. The rank-one operator $\theta_{\xi, \eta}$ defined on $X$ as
\begin{equation}
\theta_{\xi,\eta}(\zeta)=\xi\cdot\langle \eta,\zeta\rangle_A \; \; \text{for} \; \xi, \eta, \zeta\in X,
\end{equation}
is adjointable and we have $\theta_{\xi,\eta}^*=\theta_{\eta,\xi}$. Then
$\Kk(X)=\clsp\{\theta_{\xi, \eta}\mid \xi, \eta\in X\}$ is the ideal of (generalized)
compact operators in $\Ll(X)$.

\vspace{2mm}
If $\varphi:A\rightarrow\Ll(X)$ is a $*$-homomorphism, then $\varphi$ induces a left action of $A$
on a right Hilbert $A$-module $X$ given by $a\cdot\xi=\phi(a)\xi$, for $a\in A$ and $\xi\in X$.
Then $X$ becomes a right Hilbert $A$--$A$-bimodule. The standard bimodule $_AA_A$ is
equipped with $\langle a, b\rangle _A=a^*b$, and the right and left actions are simply given
by right and left multiplication in $A$, respectively.

For right Hilbert $A$--$A$-bimodules $X$ and $Y$, the (balanced) tensor product $X\otimes_AY$ becomes
a right Hilbert $A$--$A$-bimodule with the natural right action,  the left action implemented by the
homomorphism $A\ni a \mapsto\varphi(a)\otimes_A I_Y\in\Ll(X\otimes_A Y)$, and the $A$-valued
inner product given by
\begin{equation}
\langle\xi_1\otimes_A\eta_1 , \xi_2\otimes_A\eta_2\rangle_A=\langle\langle\xi_2,
\xi_1\rangle_A\cdot\eta_1, \eta_2\rangle_A,
\end{equation}
for $\xi_i\in X$ and $\eta_i\in Y$.

\vspace{2mm}
Let $P$ be a multiplicative semigroup with identity $e$, and $A$ a unital $C^*$-algebra.
For each $s\in P$ let $X_s$ be a complex vector space.
Then the disjoint union $X := \bigsqcup_{s\in P}X_{s}$ is a {\em product system} over $P$
if the following conditions hold:
\begin{enumerate}\renewcommand{\theenumi}{P\arabic{enumi}}
\item\label{it:P-1} For each $s\in P\setminus\{e\}$, $X_s$ is a right Hilbert $A$--$A$-bimodule.
\item\label{it:P-2}$X_e$ equals the standard bimodule $_AA_A$.
\item\label{it:P-3} $X$ is a semigroup such that $\xi\eta\in X_{sr}$ for $\xi\in X_s$ and $\eta\in X_r$,
and for $s, r\in P\setminus\{e\}$, this product extends to an isomorphism
$F^{s,r} : X_s\otimes_A X_r\rightarrow X_{sr}$ of right  Hilbert $A$-$A$-bimodules. If $s$ or $r$ equals $e$
then the corresponding product in $X$ is induced by the $A$-$A$-bimodule structure on the fibers.
\end{enumerate}
We denote by $\langle\cdot,\cdot\rangle_s$ the $A$-valued inner product on $X_s$,
 by  $\rho_s$ the right action of $A$ on $X_s$, and by $\varphi_s$ the homomorphism
from $A$ into $\Ll(X_s)$.  We have $\varphi_{sr}(a)(\xi\eta)=(\varphi_s(a)\xi)\eta$ for all
$\xi\in X_s$, $\eta\in X_r$, and $a\in A$.

Let $I_s$ be the identity operator in $\Ll(X_s)$ for every $s\in P$.
The product system $X$ is \emph{associative} provided that
\begin{equation}\label{ps-associative}
F^{sr, q}(F^{s, r}\otimes_A I_q)=F^{s, rq}(I_s\otimes_A F^{r, q})
\end{equation}
 for all $s, r, q\in P$, see e.g. \cite{ShaSol} or \cite{Vi}.

\begin{rmk}\rm
For $s\in P$, the multiplication on $X$ induces maps
$F^{s, e}:X_s\otimes_A X_e\rightarrow X_s$ and $F^{e, s}:X_e\otimes_A X_s\rightarrow X_s$
by multiplication $F^{s, e}(\xi\otimes a)=\xi\, a$ and $F^{e, s}(a\otimes\xi)=a\, \xi$
for $a\in A$ and $\xi\in X_s$. Note that $F^{s,e}$ is automatically an isomorphism,
but $F^{e,s}$ may not be. The latter map is an isomorphism whenever $\overline{\varphi(A)X_s}=X_s$,
in which case $X_s$ is called essential, see \cite{Fow}. If $A$ is unital and $\varphi(1)\xi=\xi$ for all
$\xi\in X_s$ then $X_s$ is essential.
\end{rmk}

For each pair $s, r\in P\setminus\{e\}$, the isomorphism
$F^{s, r} : X_s\otimes_A X_r\rightarrow X_{sr}$
allows us to define a $*$-homomorphism $i_s^{sr}:\Ll(X_{s})\to\Ll(X_{sr})$ as
\begin{equation}
i_s^{sr}(S)=F^{s, r}(S\otimes_A I_r)(F^{s, r})^*,
\end{equation}
for $S\in\Ll(X_s)$.  When $s=e$, the homomorphism $i_e^r$ defined on $\Ll(A)=A$ is given simply
by $i_e^r(a)=\varphi_r(a)$ for $a\in A$. Also, $i_s^s=I_s$ for all $s\in P$.

\vspace{2mm}
Many interesting product systems arise over semigroups equipped with additional structures.
In \cite{Ni}, $(G, P)$ is called a quasi-lattice ordered group if (i) $G$ is a discrete group, (ii) $P$ is
a subsemigroup of $G$ with $P\bigcap P^{-1}=\{e\}$, (iii)  every two elements $s,r\in G$ which have
a common upper bound in $P$ have a least upper bound $s\vee r \in P$ with respect to the order
given by $s\preceq r\; \Leftrightarrow\; s^{-1}r \in P$. If this is the case we write $s\vee r<\infty$,
otherwise $s\vee r=\infty$.

Assuming $X$ is a product system over $P$ with $(G,P)$ a quasi-lattice ordered group, there
naturally arises a certain property related to  compactness.
A product system $X=\sqcup_{s\in P}X_s$ is called {\it compactly aligned}, \cite{Fow}, if
for all $s, r\in P$ with $s\vee r<\infty$ and $S\in\Kk(X_s)$ and $T\in\Kk(X_r)$, we have
\begin{equation}
i_s^{s\vee r}(S)i_r^{s\vee r}(T)\in\Kk(X_{s\vee r}).
\end{equation}


\subsection{$C^*$-algebras associated to product systems}\label{Prelim-2}

Let $P$ be a semigroup with identity, $A$ be a unital $C^*$-algebra, and
$X=\sqcup_{s\in P}X_s$ be a product system of right Hilbert $A$--$A$-bimodules over $P$.

Let $C$ be a $C^*$-algebra. A mapping $\psi : X\longrightarrow C$ is said to be a Toeplitz
representation of $X$ if the following conditions hold:
\begin{enumerate}\renewcommand{\theenumi}{T\arabic{enumi}}
\item\label{it:T-1} for each $s\in P\setminus\{e\}$, $\psi_s:=\psi\upharpoonright_{X_s}$ is linear,
\item\label{it:T-2} $\psi_e:A\longrightarrow C$ is a $C^*$-homomorphism,
\item\label{it:T-3} $\psi_s(\xi)\psi_r(\eta)=\psi_{sr}(\xi\eta)\; \; $ for $ \; \xi\in X_s$ and $\eta\in X_r$,
\item\label{it:T-4} $\psi_s(\xi)^*\psi_r(\eta)=\psi_e(\langle\xi, \eta\rangle_s)$ for $\xi, \eta\in X_s$.
\end{enumerate}
As shown by Pimsner in \cite{P}, for each $s\in P$ there exists a corresponding $*$-homomorphism
$\psi^{(s)} : \Kk(X_s) \longrightarrow C$ such that
\begin{equation}
\psi^{(s)}(\theta_{\xi,\eta})=\psi_s(\xi)\psi_s(\eta)^*\, , \; \text{for}\; \xi,\eta\in X_s.
\end{equation}

Assume $(G,P)$ is a quasi-lattice ordered group and $X$ is compactly aligned.
In \cite{Fow}, a Toeplitz representation  $\psi$ of $X$ is said to be {\em Nica covariant} if
\begin{equation}\label{eq:Nica-cov-rep}
\psi^{(s)}(S)\psi^{(r)}(T)=
    \begin{cases}\psi^{(s\vee r)}(i^{s\vee r}_s(S)i^{s\vee r}_r(T)) &\; \text{if} \; s\vee r <\infty \\
        0 \; & \; \; \text{otherwise}
    \end{cases}
\end{equation}
for $S\in {\mathcal K}(X_s)$, $T\in {\mathcal K}(X_r)$, and $s, r\in P$.

\begin{dfn}(\cite{Fow, CLSV, BanHLR}) Let $(G,P)$ be a quasi-lattice ordered group and $X$
a compactly aligned product system  over $P$. The Nica-Toeplitz algebra  $\NT(X)$ is the
$C^*$-algebra generated by a universal Nica covariant Toeplitz representation $i_X$ of $X$.
\end{dfn}

Fix a compactly aligned product system $X$ of right Hilbert $A$--$A$-bimodules  over a
semigroup $P$ in a quasi-lattice ordered group $(G, P)$. Let $i_X$ be the universal
Nica covariant Toeplitz representation of $X$ and denote by $i_s$ the restriction of
$i_X$ to $X_s$ for  $s\in P$. Recall that $\NT(X)$ is spanned by $i_s(\xi)i_r(\eta)^*$
for $\xi\in X_s, \eta\in X_r$, and there is a gauge coaction $\delta$ of $G$ such that
$\delta(i_s(\xi))=i_s(\xi)\otimes s$, cf. \cite{Fow, SY, CLSV}. The \emph{core} of $\NT(X)$ is
the $C^*$-subalgebra $\Ff$ spanned by the monomials $i_s(\xi)i_s(\eta)^*$ for
$\xi, \eta \in X_s$ and $s\in P$. Then
\begin{equation}\label{eq:core_as_compacts}
\Ff=\clsp\{ i^{(s)}(T)\mid s\in P, \,T\in \Kk(X_s)\},
\end{equation}
see e.g.  equation (3.4) in \cite{CLSV}. Let $\Phi^\delta$ be
the conditional expectation from $\NT(X)$ onto $\Ff$ given by
\begin{equation}\label{def:condexp}
\Phi^\delta(i_s(\xi)i_r(\eta)^*)=\begin{cases}i_s(\xi)i_s(\eta)^*  &\text{ if }s=r \\ 0  &\text{ otherwise}
\end{cases}
\end{equation}
for $s, r\in P$. Recall that a finite subset $F$ of $P$ is $\vee$-closed if $r\vee s\in F$ for all $r,s\in F$.
For a finite subset $F$ of $P$ which is $\vee$-closed, \cite[Lemma 3.6]{CLSV} says that
\begin{equation}\label{def:B_F}
B_F:=\{\sum_{s\in F} i^{(s)}(T_s)\mid T_s \in \Kk(X_s)\}
\end{equation}\label{def:B_B}
is a $C^*$-subalgebra of $\Ff$, and we have $\Ff=\overline{\bigcup_F B_F}$.


\subsection{The Fock representation and Nica covariance}\label{Prelim-3}

Let $X$ be a  product system over $P$ of right Hilbert $A$--$A$-bimodules and let $l:X\to \Ll(F(X))$
be the Fock representation of $X$ constructed in \cite[page 340]{Fow}. We use the notation of
\cite{Vi} to describe $l$: the restriction of $l$ to $X_s$ is given by $l_s(\xi)\eta=F^{s,r}(\xi\otimes_A \eta)$
if $\xi\in X_s$ and $\eta \in X_r$ for $r\in P$. The adjoint acts by
$$
l_s(\eta)^*\zeta=\begin{cases} \varphi_{s^{-1}r}(\langle \eta, \zeta'\rangle_s)\zeta^{''}
&\text{ if } r\in sP \text{ and }\zeta=F^{s, s^{-1}r}(\zeta'\otimes_A \zeta^{''})\\
0  &\text{ if }  r\notin sP.
\end{cases}
$$
Let $\xi,\eta\in X_s$ for $s\in P$, and $\zeta\in X_r$ for $r\in P$. If $r\notin sP$, then
$l_s(\xi)l_s(\eta)^*\zeta=0$. If $r\in sP$ then we have
\begin{align*}
l_s(\xi)l_s(\eta)^*\zeta
&=F^{s, s^{-1}r}(\xi\otimes_A \varphi_{s^{-1}r}(\langle \eta, \zeta'\rangle_s)\zeta^{''})\\
&=F^{s, s^{-1}r}(\theta_{\xi, \eta}(\zeta')\otimes_A \zeta^{''})\\
&=F^{s, s^{-1}r}(\theta_{\xi, \eta}\otimes_A I_{s^{-1}r})(\zeta'\otimes_A \zeta^{''})\\
&=i_s^r(\theta_{\xi,\eta})\zeta.
\end{align*}

It was asserted in \cite[\S 4]{SY} that for $(G, P)$ quasi-lattice ordered and $X$  compactly aligned
Fowler had proved that $l$ is Nica covariant as in (\ref{eq:Nica-cov-rep}). However,
although one may use \cite[Propositions 5.6 and 5.9]{Fow} to deduce
this claim, there is no such explicit  result in \cite{Fow}.

It is an instructive exercise to see how the quasi-lattice ordered property of $(G, P)$ almost imposes
the condition on $X$ being compactly aligned, and consequently makes $l$ Nica covariant as a
representation into the $C^*$-algebra $\Ll(F(X))$. Indeed, let $\theta_{\xi, \eta}\in \Kk(X_s)$ and
$\theta_{z,w}\in \Kk(X_r)$ for $s,r\in P$. What can be said
of the element $K_{s,r}:=l^{(s)}(\theta_{\xi, \eta})l^{(r)}(\theta_{z,w})$ in $\Ll(F(X))$?
If $\zeta\in X_q$ then $K_{s,r}\zeta=0$ unless $q\in sP\cap rP$ or, equivalently, $s\vee r<\infty$
and $q\in (s\vee r)P$. Thus for $\zeta\in X_{s\vee r}$ we have
$K_{s,r}\zeta=i_s^{s\vee r}(\theta_{\xi, \eta})i_r^{s\vee r}(\theta_{z,w})\zeta$. Now, if
$i_s^{s\vee r}(\theta_{\xi, \eta})i_r^{s\vee r}(\theta_{z,w})=\theta_{x,y}$ for some
$x,y\in X_{s\vee r}$, it follows  that $K_{s,r}\zeta=l^{(s\vee r)}(\theta_{x,y})\zeta$. By linearity and continuity,
 $$
K_{s,r} =  l^{(s)}(\theta_{\xi, \eta})l^{(r)}(\theta_{z,w})=
l^{(s\vee r)}(i_s^{s\vee r}(\theta_{\xi, \eta})i_r^{s\vee r}(\theta_{z,w}))
 $$
 whenever $i_s^{s\vee r}(\theta_{\xi, \eta})i_r^{s\vee r}(\theta_{z,w})\in \Kk(X_{s\vee r})$.
So if  $X$ is compactly aligned  in Fowler's sense, i.e. if $i_s^{s\vee r}(S)i_r^{s\vee r}(T)\in
\Kk(X_{s\vee r})$ whenever $S\in \Kk(X_s)$ and $T\in \Kk(X_r)$, then $l$  is Nica covariant, since
\begin{equation}\label{eq:l-Nica-cov}
l^{(s)}(S)l^{(r)}(T)=l^{(s\vee r)}(i_s^{s\vee r}(S)i_r^{s\vee r}(T)).
\end{equation}
When $s\vee r=\infty$ identity \eqref{eq:l-Nica-cov} holds as well,  because then both sides are $0$.

By the universal property of $\NT(X)$  there is a homomorphism
\begin{equation}\label{def:l_star}
l_*:\NT(X)\to \Ll(F(X))
\end{equation}
such that $l_*(i_s(\xi))=l_s(\xi)$ for all $s\in P$ and $\xi \in X_s$.


\section{KMS states on the Nica-Toeplitz algebra of product systems}\label{section:KMS_product_systems}

KMS states on the Toeplitz algebra associated to a single bimodule were studied in many contexts,
and a general unified approach was obtained in \cite{Lac-Nes}.
In the present paper, we aim to analyze KMS states in the context of product systems of right Hilbert
bimodules. We begin by introducing a certain type of dynamics $\sigma_t$, $t\in \RR$, on
the algebra $\NT(X)$ for an arbitrary compactly aligned product system $X$ over $P$ in case $(G, P)$ is
a quasi-lattice ordered group. Our construction is analogous to quasi-free dynamics
on Cuntz-Pimsner algebras considered in \cite{Z} and \cite{Lac-Nes}.

Later we introduce a class of compactly aligned product systems of finite type over a lattice semigroup
 $P$ and analyze KMS states corresponding to certain natural dynamics. Our characterizations of the
KMS$_\beta$ and the ground states in terms of certain states
of the core $\Ff$ of $\NT(X)$ are very much in the spirit of Laca's work \cite{Lac}, see also
\cite{Lac-Nes2}.


\subsection{Compactly aligned product systems over a quasi-lattice ordered group.}
\begin{prop}\label{prop:dynamics_on_ps}
Let $(G,P)$ be a quasi-lattice ordered group and $X$ a compactly aligned product system over $P$
of right Hilbert $A$--$A$-bimodules. Assume that for each $s\in P$ there is a strongly continuous
one-parameter group $t\to U_t^{(s)}$ in $\Ll(X_s)$, $t\in \RR$, such that $U_t^{(s)}$ commutes
with $\varphi_s(A)$, $U_t^{(s)}$ is unitary for all $t$, $U_0^{(s)}=I_s$, and
\begin{equation}\label{eq:one_param_gen_TX}
U_t^{(s)}(\varphi_s(a)\xi)=\varphi_s(a)(U_t^{(s)}\xi),
\end{equation}
for all $a\in A$, $s\in P$, $\xi\in X_s$, and $t\in \RR$. If, in addition,
the isomorphisms  $F^{s,r}:X_s\otimes_A X_r\to X_{sr}$ satisfy
\begin{equation}\label{eq:UcompatibleX}
 F^{s,r}\circ (U_t^{(s)}\otimes_AU_t^{(r)})=U_t^{(sr)}\circ F^{s,r}
\end{equation}
for all $s,r\in P$, then there is a one-parameter group of automorphisms
$t\to \sigma_t\in \operatorname{Aut}(\NT(X))$ such that
\begin{equation}\label{eq:induced_dynamics}
\sigma_t(i_e(a))=i_e(a) \quad \text{ and }\quad \sigma_t(i_s(\xi))=i_s(U_t^{(s)}\xi)
\end{equation}
 for all $a\in A$, $\xi\in X_s$ and $s\in P$.
\end{prop}

\begin{proof}
Note that $U_t^{(s)}\otimes_AU_t^{(r)}$ is a well-defined operator on $X_s\otimes_A X_r$
due to the assumption that $U_t^{(q)}$ and  $\varphi_q(A)$ commute for every $q\in P$.
Roughly,  \eqref{eq:UcompatibleX} says that the one-parameter unitary group is compatible with
the product system, and ensures that the one-parameter group $\{U_t^{(s)}\}_{t\in \RR, s\in P}$
combine to give a dynamics on $\NT(X)$. To see this, we define a new representation $\psi$ of $X$
in $\NT(X)$ by $\psi_e(a)=i_e(a)$ for $a\in A$ and $\psi_s(\xi)=i_s(U_t^{(s)}\xi)$
for $\xi \in X_s$ and $s\in P$. Condition \eqref{eq:one_param_gen_TX} shows that $(\psi_e,\psi_s)$ is a
Toeplitz representation of $X_s$ for all $s\in P$, and condition \eqref{eq:UcompatibleX} implies that
$\psi_{sr}(\xi\eta)=\psi_s(\xi)\psi_r(\eta)$ for all $\xi \in X_s$, $\eta\in X_r$, $s, r\in P$.

We next prove that $\psi$ is Nica covariant, from which by applying  the universal
property of $\NT(X)$ we deduce the existence of $*$-homomorphisms $\sigma_t$, $t\in\RR$,
as postulated in \eqref{eq:induced_dynamics}.

Note first that  $\psi^{(s)}(\theta_{\xi,\eta})=i^{(s)}(\theta_{U_t^{(s)}\xi,
U_t^{(s)}\eta})=i^{(s)}(U_t^{(s)}\theta_{\xi,\eta}U_t^{(s)*})$, for all $s\in P$
and $\theta_{\xi,\eta}\in\Kk(X_s)$. By continuity of all maps involved we have
\begin{equation}\label{eq:psi_from_U}
\psi^{(s)}(S)=i^{(s)}(U_t^{(s)}SU_t^{(s)*})
\end{equation}
for all $S\in \Kk(X_s)$ and $s\in P$.

Let $S\in \Ll(X_s)$ and $T\in \Ll(X_r)$ for $s, r\in P$. We aim to prove that
\begin{equation}\label{eq:existence_of_dynamics}
U_t^{(s\vee r)}(i_s^{s\vee r}(S)i_r^{s\vee r}(T))U_t^{(s\vee r)*}
=i_s^{s\vee r}(U_t^{(s)}SU_t^{(s)*})i_r^{s\vee r}(U_t^{(r)}TU_t^{(r)*}).
\end{equation}
For this it suffices to show that
\begin{equation}\label{eq:help_existence_of_dynamics}
U_t^{(s\vee r)}i_s^{s\vee r}(S)=i_s^{s\vee r}(U_t^{(s)}SU_t^{(s)*})U_t^{(s\vee r)},
\end{equation}
since taking adjoints in \eqref{eq:help_existence_of_dynamics} and replacing $S^*$ with
$T$ proves \eqref{eq:existence_of_dynamics}. Now, write $r'=s^{-1}(s\vee r)$. The left-hand side of \eqref{eq:help_existence_of_dynamics} can be transformed as follows:
\begin{align}
U_t^{(s\vee r)}i_s^{s\vee r}(S)
&=U_t^{(s\vee r)}F^{s, r'}(S\otimes_A I_{r'})(F^{s, r'})^*\notag\\
&=F^{s, r'}(U_t^{(s)}\otimes_A U_t^{(r')}) (S\otimes_A I_{r'})(F^{s, r'})^*
\;\;\; \text{by} \;\; \eqref{eq:UcompatibleX} \notag\\
&=F^{s, r'}(U_t^{(s)}S\otimes_A U_t^{(r')}) (F^{s, r'})^*\notag\\
&=F^{s, r'}(U_t^{(s)}SU_t^{(s)*}\otimes_A I_{r'})(U_t^{(s)}\otimes
U_t^{(r')})(F^{s, r'})^* \label{eq:middle_to_dynamics}.
\end{align}
Now, inserting a factor $(F^{s, r'})^*(F^{s, r'})$ after $(U_t^{(s)}SU_t^{(s)*}\otimes_A I_{r'})$,
grouping terms and using \eqref{eq:UcompatibleX} once more  shows that
\eqref{eq:middle_to_dynamics} equals $i_s^{s\vee r}(U_t^{(s)}SU_t^{(s)*})
U_t^{(s\vee r)}$, as claimed in \eqref{eq:help_existence_of_dynamics}. Now, if $S\in \Kk(X_s)$ and
$T\in \Kk(X_r)$, we have $U_t^{(s)}SU_t^{(s)*}\in \Kk(X_s)$
and $U_t^{(r)}TU_t^{(r)*}\in \Kk(X_r)$. Further, since $X$ is compactly aligned,
we also have $i_s^{s\vee r}(S)i_r^{s\vee r}(T)\in \Kk(X_{s\vee r})$ as well as
$$
i_s^{s\vee r}(U_t^{(s)}SU_t^{(s)*})i_r^{s\vee r}(U_t^{(r)}TU_t^{(r)*})\in \Kk(X_{s\vee r})
$$
for $s\vee r<\infty$. In this case, we use  Nica covariance of $i_X$
to deduce that
\begin{align*}
\psi^{(s)}(S)\psi^{(r)}(T)
&=i^{(s)}(U_t^{(s)}SU_t^{(s)*})i^{(r)}(U_t^{(r)}T_t^{(r)*})
\;\;\;\;\; \text{ by }
\eqref{eq:psi_from_U}\\
&=i^{(s\vee r)}(i_s^{s\vee r}(U_t^{(s)}SU_t^{(s)*})i_r^{s\vee r}(U_t^{(r)}TU_t^{(r)*}))\\
&=i^{(s\vee r)}(U_t^{(s\vee r)}(i_s^{s\vee r}(S)i_r^{s\vee r}(T))U_t^{(s\vee r)*})
\;\;\; \text{ by }\eqref{eq:existence_of_dynamics}\\
&=\psi^{(s\vee r)}(i_s^{s\vee r}(S)i_r^{s\vee r}(T))
\;\;\;\;\;\; \text{ by }\eqref{eq:psi_from_U}.
\end{align*}
In the case of $s\vee r=\infty$,  we have $i_s^{s\vee r}(S)i_r^{s\vee r}(T)=0$ which
proves the required Nica covariance of $\psi$.

Consequently, each $\sigma_t$ is a $*$-homomorphism satisfying (\ref{eq:induced_dynamics}).
Since (\ref{eq:induced_dynamics}) immediately implies that $\sigma_0=\id$ and
$\sigma_{t+t'}=\sigma_t\sigma_{t'}$, we see that $\sigma:\RR\to\operatorname{Aut}(\NT(X))$ is the required
one-parameter automorphism group.
\end{proof}

We next recall the notions of KMS$_\beta$ state,  KMS$_\infty$ state, and ground state. Nowadays one
often employs definitions of ground state and KMS$_\beta$ state which are different from, although
of course equivalent to, the more classical ones in \cite{Bra-Rob} and \cite{Ped} (see Section 7 in
\cite{Lac-Rae2} for explanatory details).

Given a $C^*$-algebra $C$ and a homomorphism (a dynamics) $\sigma:\RR\to \Aut(C)$, an
element $c\in C$ is called $\sigma$-\emph{analytic} provided that $\RR\ni t\mapsto \sigma_t(c)$ extends to an
entire function on $\CC$. The set of $\sigma$-analytic elements is dense in $C$, see  \cite[\S 8.12]{Ped}.
For $\beta\in (0, \infty)$, a \emph{KMS$_\beta$ state} of a dynamical system $(C,\sigma)$ is defined as a state
$\omega$ of $C$ which satisfies the KMS$_\beta$ condition
\begin{equation}\label{def:KMS-beta}
\omega(cd)=\omega(d\sigma_{i\beta}(c))
\end{equation}
for all $\sigma$-analytic $c,d$  in $C$. It is known that it suffices to have \eqref{def:KMS-beta}
satisfied for a subset of $\sigma$-analytic elements of $C$ which spans a dense subspace of $C$,
\cite[Proposition 8.12.3]{Bra-Rob}. A state $\omega$ of $C$ is said to be a  {\em ground state} of $(C, \sigma)$
if for every $\sigma$-analytic $c,d$  in $C$, the entire function $\CC\ni z\mapsto \omega(c\sigma_z(d))$ is
bounded on the upper-half plane. Again,  it is known that it suffices to have boundedness for a set of elements
which spans a dense subspace of the $\sigma$-analytic elements.
More recently, the notion of {\em KMS$_\infty$ state}  of $(C,\sigma)$ was coined down by Connes and
Marcolli  in \cite{Con-Mar} and refers to a state which can be realized as a weak$^*$-limit of KMS$_\beta$ states
as $\beta$ runs over a net converging to $\infty$.

Now, we are defining the type of dynamics on a product system we will use throughout
the remining part of this paper, as follows.
Fix a quasi-lattice ordered group $(G, P)$ and a  compactly aligned product system $X$ over $P$.
We are going to define a dynamics arising from a homorphism of G into $(0, \infty)$. Suppose
that $N:G\to (0,\infty)$ is a multiplicative homomorphism. For every $s\in P$ and $t\in \RR$
define $U_t^{(s)}\xi=N(s)^{it}\xi$ in $\Ll(X_s)$. Since $N$ is multiplicative, the family
$\{U_t^{(s)} : s\in P, t\in \RR\}$ satisfies the conditions of Proposition~\ref{prop:dynamics_on_ps}.
Hence there is a dynamics $\sigma^N$ on $\NT(X)$ such that
\begin{equation}\label{def:sigmaN}
\sigma_t^N(i_e(a))=i_e(a) \text{ and } \sigma_t^N(i_s(\xi))=N(s)^{it}i_s(\xi)
\end{equation}
for all $a\in A$, $\xi\in X_s$, $s\in P$. A routine proof of the following lemma is omitted.

\begin{lemma} The spanning elements $\{i_s(\xi)i_r(\eta)^* : \xi\in X_s, \eta\in X_r , s,r\in P\}$ of
$\NT(X)$ are $\sigma^N$-analytic.
\end{lemma}

We aim to establish analogues of \cite[Theorem 12]{Lac}. As we shall see, the degree
of sharpness of our results depends on assumptions on the product system $X$, so we separate the
characterizations of KMS$_\beta$ states, KMS$_\infty$ states, and ground states.

In one direction, for an arbitrary $\NT(X)$, a dynamics $\sigma^N$ where $N$ is an injective
homomorphism, and for every $0<\beta<\infty$, KMS$_\beta$ states turn out to be lifted from
tracial states of $\Ff$ with a scaling property, as shown in Proposition \ref{prop:KMS-state-on-F} below.
The non-trivial converse
will be proved in Theorem~\ref{thm:Laca-12-KMSbeta}, under additional hypotheses  on $X$. Recall that there
is a conditional expectation $\Phi^\delta : \NT(X) \rightarrow\Ff$  given by
$$
\Phi^\delta(i_s(\xi)i_r(\eta)^*)=\begin{cases}i_s(\xi)i_s(\eta)^*  &\text{ if }s=r \\ 0  &\text{ otherwise}\end{cases}
$$
for $s, r\in P$,  as in \eqref{def:condexp}.

\begin{prop}\label{prop:KMS-state-on-F}
Let $(G, P)$ be a quasi-lattice ordered group and $X$ a compactly aligned product system
over $P$. Let $\sigma^N$ be the dynamics on $\NT(X)$ arising from an injective homomorphism
$N:G\to (0,\infty)$. Let $0<\beta<\infty$. If $\omega$ is a KMS$_\beta$ state of $\NT(X)$, then
$\omega$ factors through $\Phi^\delta$ to give a tracial state $\phi$ on $\Ff$ that satisfies
the scaling identity
\begin{equation}\label{scaling}
\phi(i_s(\xi)yi_r(\eta)^*)=\delta_{s,r}N(s)^{-\beta}\phi(y\langle\eta,\xi\rangle_s)
\end{equation}
for all $\xi\in X_s$, $\eta\in X_r$, $s, r\in P$ and $y\in \Ff$.
\end{prop}

\begin{proof}
Since all elements of the core $\Ff$ are fixed by $\sigma^N$, it follows that the restriction
of $\omega$ to $\Ff$ is a trace, \cite{Bra-Rob}.

Now let $i_s(\xi)i_r(\eta)^*$ be  a spanning element
of $\NT(X)$ with $\xi\in X_s$, $\eta\in X_r$, $r,s\in P$, and note that by twice applying
the KMS$_\beta$ condition we get
$$
\omega(i_s(\xi)i_r(\eta)^*)=N(sr^{-1})^{-\beta}\omega(i_s(\xi)i_r(\eta)^*).
$$
Since $N$ is injective, $\omega(i_s(\xi)i_r(\eta)^*)=0$ unless $s=r$. In other
words, $\omega$ factors through $\Phi^\delta$ to give a trace $\phi$ on $\Ff$ such that
$\phi\circ \Phi^\delta=\omega$. It is then immediate from  the KMS$_\beta$ condition 
that $\phi$ satisfies the scaling identity \eqref{scaling}.
\end{proof}

We next move to characterize ground states. For $s\in P$, recall that
$B_s:=B_{\{s\}}=\{i^{(s)}(T_s)\mid T_s \in \Kk(X_s)\}$ is a $C^*$-subalgebra of
$\Ff$, as in \eqref{def:B_F}.

\begin{theorem}\label{thm:Laca-12-ground}
Let $(G, P)$ be a quasi-lattice ordered group and $X$ a compactly aligned product
system over $P$. Let $\sigma^N$ be the dynamics on $\NT(X)$ arising from a homomorphism
$N:G\to (0,\infty)$ such that $N(r)\geq 1$ for $r\in P$ and equality holds only when $r=e$.
Then $\phi\mapsto \phi\circ \Phi^\delta$ is an affine isomorphism from
the states of $\Ff$ such that $\phi\upharpoonright_{B_s}=0$ for all $s\in P\setminus \{e\}$ onto the
ground states of $\NT(X)$.
\end{theorem}

\begin{proof}
Let $\phi$ be a state of $\Ff$ which is zero on $B_s$ for $s>e$. Set $\omega:=\phi\circ \Phi^\delta$.
Let $y$ be arbitrary and $y'=i_s(\xi)i_r(\eta)^*$ an $\sigma^N$-analytic element in $\NT(X)$.
We must show that the function $F(z):= \omega(y\sigma^N_z(y'))$ is bounded
on the upper-half plane. Since $F(z)=N(sr^{-1})^{iz}\omega(yy')$, this function is bounded on
the upper-half plane in case $r=e$ because $N(s)\geq 1$. If $r>e$, an application of
the Cauchy-Schwarz inequality (where we let $y_1=yi_s(\xi)$) gives
$$
\vert \omega(yy')\vert\leq
\omega(y_1y_1^*)^{1/2}\omega(i_r(\eta)i_r(\eta)^*)^{1/2}.
$$
Since $i_r(\eta)i_r(\eta)^*\in B_r$, the assumption on $\phi$ implies that
$\omega(yy')$, and hence $F(z)$ vanish, proving the required boundedness.

Conversely, if $\omega$ is a ground state of $\NT(X)$, then boundedness on the upper-half
plane   of the function $z\mapsto \omega(y\sigma_z^N(y'))$ for arbitrary $y$ and
$y'=i_r(\eta)^*$ with $r>e$  forces $\omega(y y')$ to be $0$. Hence $\omega$ vanishes on
$B_r$ for any $r>e$, as claimed.

The correspondence $\phi\mapsto \phi\circ \Phi^\delta$ is an affine map by the same argument as in
the proof of \cite[Theorem 12]{Lac}.
\end{proof}


\subsection{Product systems of finite type over a lattice ordered group}\label{section_finite-type}
Throughout this section, $A$ is a unital $C^*$-algebra and $(G, P)$  denotes a {\em lattice ordered group} as in
\cite[Definition 1]{Lac}. That is: $P$ is an abelian
cancellative  semigroup with identity $e$,  $G=PP^{-1}$ is the Groethendieck enveloping group
of $P$, the semigroup $P$ satisfies $P\cap P^{-1}=\{e\}$, and every pair of elements $s,r\in G$ has a (unique)
least common upper bound $s\vee r\in G$ with respect to the partial order given by $s\leq r\iff s^{-1}r\in P$.
Note that for $s, r\in P$, the element $(s\vee r)^{-1}sr$ is their
greatest lower bound, which we denote  $s\wedge r$. The properties of $s\vee r$ and $s\wedge r$ that
are relevant to our analysis are contained in \cite[Lemma 2]{Lac}.

\begin{dfn}\label{def:finite-type}
Let $(G, P)$ be a lattice ordered group and let  $X$ be a product system over $P$ of right Hilbert $A$--$A$-bimodules.
We say that $X$ is \emph{of finite type} if for every $s\in P$ there are $N_s\in \NN$ and elements
$\{\31_0^s,\dots , \31_{N_s-1}^s\}\in X_s$  such that:
\begin{enumerate}
\item $X_s=\left\{\sum_{j=0}^{N_s-1} \varphi_s(a_j)\31^s_j: a_j\in A, j=0,\dots ,N_s-1\right\}$,
\item $X_s=\left\{\sum_{j=0}^{N_s-1} \rho_s(b_j)\31^s_j: b_j\in A, j=0,\dots ,N_s-1\right\}$,
\item $\langle \31^s_j,\31^s_k\rangle_s=\delta_{j, k}$ for $s\in P$, $j,k\in \{0,\dots ,N_s-1\}$, and
\item\label{F-1s1t}
for every $s,r\in P$ there is a map $$\mathfrak{m}_{s,r}:\{0, \dots, N_s-1\}\times \{0, \dots ,N_r-1\}
\to \{0,\dots, N_{sr}-1\}$$ such that
\begin{equation}\label{dotdefinition}
F^{s,r}(\31^s_j\otimes_A \31^r_k)=\31^{sr}_{\mathfrak{m}_{s,r}(j,k)}
\end{equation}
for $j\in\{0, \dots, N_s-1\}$ and $k\in\{0, \dots ,N_r-1\}$.
\end{enumerate}

For $j\in\{0, \dots, N_s-1\}$ and $k\in\{0, \dots ,N_r-1\}$ we will often write $j\cdot k$ for the element
$\mathfrak{m}_{s,r}(j, k)$  in $\{0,\dots, N_{sr}-1\}$. Therefore formula (\ref{dotdefinition}) becomes
$F^{s,r}(\31^s_j\otimes_A \31^r_k)=\31^{sr}_{j\cdot k}$.

In case there is $Z\in A$ such that $\31^s_j=\varphi_s(Z^j)\31_0^s$ for all $j\in\{0, \dots, N_s-1\}$ and
all $s\in P$, we say that $X$ is \emph{singly generated}. We then write $\31_s:=\31^s_0$.
\end{dfn}

Conditions (2) and (3)  in Definition~\ref{def:finite-type} say that the right Hilbert
$A$-module $X_s$ has an orthonormal basis $\{\31_0^s,\dots , \31_{N_s-1}^s\}$.
Condition (1) in Definition~\ref{def:finite-type} then expresses the fact that this basis also
generates $X_s$ as a left $A$-module, and condition (4) says that these bases for the modules
$X_s$, $s\in P$, are coherent with respect to the multiplication in the product system.
We note that the tensor product of the orthonormal bases $\{\31^s_j\}_{j=0, \dots,N_s-1}$
for $X_s$ and $\{\31^r_k\}_{k=0, \dots ,N_r-1}$ for $X_r$ will be an orthonormal basis for
$X_s\otimes_A X_r$ when $X_r$ is essential, see \cite[Proof of Proposition 4.2]{Lar-Rae}.
Thus, if $X_r$ are essential for all $r\in P$, the maps $\mathfrak{m}_{s,r}$ are bijective for all $r,s\in P$.

Conditions (2) and (3) in Definition~\ref{def:finite-type} also imply that each $\xi\in X_s$
has a unique representation, known as the {\em reconstruction formula}, given by
\begin{equation}\label{ONB}
\xi = \sum_{j=0}^{N_s-1}\31_j^s\cdot\langle\31_j^s, \xi\rangle.
\end{equation}
Note that when $X$ is a product system of finite type, $\{\theta_{\31^s_j,\31^s_j}: j=0, \dots, N_s-1\}$
is a family of mutually orthogonal self-adjoint projections in $\Kk(X_s)$ for every $s\in P$. Since
equation \eqref{ONB} implies that
 \begin{equation}\label{eq:Is-compact}
 I_s:=\sum_{j=0}^{N_s-1}\theta_{\31^s_j,\31^s_j}
 \end{equation}
for every $s\in P$, $I_s\in \Kk(X_s)$ for every $s\in P$, hence $\Ll(X_s)=\Kk(X_s)$, showing that the
left action is by compact operators in every fibre. By \cite[Proposition 5.8]{Fow}, a product system $X$
of finite type is therefore compactly aligned.

\begin{example}\label{ex-affine-toeplitz}
\rm Let $X$  be the product system over $\NN^\times$ with fibers isomorphic to the Toeplitz algebra
$\Tt$ from \cite[\S 6]{BanHLR}. The right action on fibers is implemented by an action $\beta:\NN^\times
\to \operatorname{End}(\Tt)$, and the inner products are defined via an action of
transfer operators for $\beta$.  One can verify that $X$ is associative. Using
\cite[Proposition 6.3]{BanHLR}, one can show that $X$ is of finite type with $N_m=m$ for $m\in \NN^\times$.
In fact, $X$ is even singly generated, where $Z$ is the generating non-unitary isometry in $\Tt$.
\end{example}

\begin{example}\label{ex-additive-toeplitz}
\rm Let $X$ be the product system over $\NN^\times$ with fibers isomorphic to $C(\TT)$ from
\cite[\S 5]{BanHLR} and \cite{HLS}. The right action in each $X_n= C(\TT)$ is implemented by the
endomorphism  $\alpha_n(f):z\mapsto f(z^n)$ of $C(\TT)$ for $n\in \NN^\times$. The inner product is
given by $\langle f,g \rangle_n=L_n(f^*g)$ for the transfer operator $L_n$  naturally associated with
$\alpha_n$. A routine calculation shows that $X$ is associative. That $X$ is singly generated,
with $Z$ the identity function on $\TT$, follows from the results of \cite{HLS} or \cite{BanHLR}.
\end{example}

The following result is the promised converse to Proposition~\ref{prop:KMS-state-on-F}.

\begin{theorem}\label{thm:Laca-12-KMSbeta}
Let $(G, P)$ be a lattice ordered group and $X$ a
compactly aligned product system of finite type over $P$.
Let $\sigma^N$ be the dynamics on $\NT(X)$ arising from an injective homomorphism
$N:G\to (0,\infty)$. Let $0<\beta<\infty$.

If $\phi$ is a tracial state of $\Ff$ such that
\begin{equation}\label{scaling-converse}
\phi(i_s(\31^s_j)yi_r(\31^r_l)^*)=\delta_{s,r}\delta_{j,l}N(s)^{-\beta}\phi(y)
\end{equation}
for all $y\in \Ff$, $s,r\in P$ and $j, l\in\{0, \dots ,N_s-1\}$, then $\omega:=\phi\circ \Phi^\delta$
is a KMS$_\beta$ state of $\NT(X)$.
\end{theorem}

\begin{proof}
Let $\phi$ be a tracial state on $\Ff$ satisfying the scaling identity \eqref{scaling-converse}.
Let $\omega=\phi\circ \Phi^\delta$. It suffices to verify the KMS$_\beta$ condition for $\omega$
on  $\sigma^N$-analytic elements $y_1=i_s(\varphi_s(a)\31^s_l)i_r(\varphi_r(b)\31^r_k)^*$
and $y_2=i_g(\varphi_g(a')\31^g_n)i_h(\varphi_h(b')\31^h_m)^*$ from the spanning set
of $\NT(X)$, where $l,k=0, \dots ,N_s-1$ and $n,m=0, \dots, N_r-1$. We must prove that
\begin{equation}\label{get-omega-KMS}
\omega(y_1y_2)=N(sr^{-1})^{-\beta}\omega(y_2y_1).
\end{equation}
By \cite[Proposition 5.10]{Fow}, the element $i_r(\varphi_r(b)\31^r_k)^*i_g(\varphi_g(a')\31^g_n)$
can be approximated from the span of elements of the form
$i_{r^{-1}(r\vee g)}(\xi)i_{g^{-1}(r\vee g)}(\eta)^*$, and so the definition of $\Phi^\delta$ in
\eqref{def:condexp} implies that $\omega(y_1y_2)=0$ unless $sr^{-1}(r\vee g)(g^{-1}(r\vee g)h)^{-1}=e$,
or equivalently, unless $sr^{-1}gh^{-1}=e$ in $G$.

Thus we assume $sg=rh$, and we therefore have $y_1y_2\in \Ff$. The scaling identity
\eqref{scaling-converse} implies that
\begin{equation}\label{omega-from-left}
N(r)^{-\beta}\omega(y_1y_2)\delta_{i,j}=\phi(i_r(\31^r_i)y_1y_2i_r(\31^r_j)^*).
\end{equation}
Next use  property (1) in Definition~\ref{def:finite-type} to write
\begin{align*}
i_r(\31^r_i)i_e(a)
&=\sum_{\nu=0}^{N_{r}-1}i_r(\varphi_r(a_\nu)\31^r_\nu)=
\sum_{\nu=0}^{N_{r}-1}i_e(a_\nu)i_r(\31^r_\nu)\text{ and}\\
i_r(\31^r_j)i_e(b')&= \sum_{\mu=0}^{N_{r}-1} i_r(\varphi_r(b_\mu)\31^r_\mu)=
\sum_{\mu=0}^{N_{r}-1} i_e(b_\mu)i_r(\31^r_\mu).
\end{align*}
Then the term under $\phi$ in the right-hand side of \eqref{omega-from-left} is a product of
\begin{equation}
i_r(\31^r_i)y_1
=\sum_{\nu=0}^{N_{r}-1}i_e(a_\nu)i_{rs}(\31^{rs}_{\nu\cdot l})
i_r(\31^r_k)^*i_e(b)^*\label{E1-almost-there}
\end{equation}
and
\begin{equation}
y_2i_r(\31^r_j)^*
=i_e(a')i_g(\31^g_n)\left( \sum_{\mu=0}^{N_{r}-1} i_e(b_\mu)i_{rh}(\31^{rh}_{\mu\cdot m})\right)^*.
\label{E2-almost-there}
\end{equation}
Fix $\alpha$ in $\{0,\dots ,N_s-1\}$. Inserting $i_s(\31^s_\alpha)^*i_s(\31^s_\alpha)=
\langle \31^s_\alpha,\31^s_\alpha\rangle_s=1$ between $y_1$ and $y_2$, and using the fact
that $rh=sg$ we conclude that $i_r(\31^r_i)y_1y_2i_r(\31^r_j)^*$ can be split as the product of
two terms in $\Ff$, one belonging to $B_{rs}$ and the other to $B_{sg}$. Applying the trace
property of $\phi$ on $\Ff$ and the scaling identity \eqref{scaling-converse} with $\31^s_\alpha$
shows, via \eqref{E1-almost-there} and \eqref{E2-almost-there}, that the right-hand side of
\eqref{omega-from-left} is equal to
\begin{equation}\label{eq:more-on-phi}
N(s)^{-\beta}\phi\left(i_e(a')i_g(\31^g_n)R^*Q i_r(\31^r_k)^*i_e(b)^*\right),
\end{equation}
where  $R= \sum_{\mu=0}^{N_{r}-1} i_e(b_\mu)i_{rh}(\31^{rh}_{\mu\cdot m})$ and
$Q=\sum_{\nu=0}^{N_{r}-1}i_e(a_\nu)i_{rs}(\31^{rs}_{\nu\cdot l})$ are products of linear combinations
in $X_{rh}$ and $X_{rs}$. Since
\begin{align}
R^*Q & = \left( \sum_{\mu=0}^{N_{r}-1} i_e(b_\mu)i_{rh}(\31^{rh}_{\mu\cdot m})\right)^*
\left(\sum_{\nu=0}^{N_{r}-1}i_e(a_\nu)i_{rs}(\31^{rs}_{\nu\cdot l})\right)\notag\\
&=\left( \sum_{\mu=0}^{N_{r}-1} i_e(b_\mu)i_{r}(\31^{r}_\mu)i_h(\31^h_m)\right)^*
\left(\sum_{\nu=0}^{N_{r}-1}i_e(a_\nu)i_{r}(\31^{r}_{\nu})i_s(\31^s_l)\right)\notag\\
&=i_h(\31^h_m)^* \left( \sum_{\mu=0}^{N_{r}-1} i_e(b_\mu)i_{r}(\31^{r}_\mu)\right)^*
\left(\sum_{\nu=0}^{N_{r}-1}i_e(a_\nu)i_{r}(\31^{r}_{\nu})\right)i_s(\31^s_l)\notag\\
&=i_h(\31^h_m)^* (i_r(\31^r_j)i_e(b'))^*i_r(\31^r_i)i_e(a) i_s(\31^s_l)\notag\\
&=i_h(\31^h_m)^* i_e(b')^* i_r(\31^r_j)^*i_r(\31^r_i)i_e(a) i_s(\31^s_l),\label{eq:even-more-on-phi}
\end{align}
we get $\phi(i_r(\31^r_i)y_1y_2i_r(\31^r_j)^*)=N(s)^{-\beta}\phi(y_2i_e(\langle \31^r_j,\31^r_i\rangle_r) y_1)$.
Hence formula \eqref{omega-from-left} now takes the form
$$ N(r)^{-\beta}\omega(y_1y_2)\delta_{i,j}=N(s)^{-\beta}\phi(y_2i_e(\langle \31^r_j,\31^r_i\rangle_r) y_1).  $$
Both sides are equal to $0$ when $i\neq j$, and with $i=j$ we have
$$
\omega(y_1y_2)=N(r)^\beta \phi(i_r(\31^r_i)y_1y_2 i_r(\31^r_i)^*)=N(sr^{-1})^{-\beta}\omega(y_2y_1),
$$
which is exactly our claim \eqref{get-omega-KMS}. This completes the proof of the theorem.
\end{proof}

Since the trace property is preserved under weak$^*$-limits we obtain the following necessary condition
on a ground state to be a KMS$_\infty$ state.

\begin{cor}\label{characterise-KMSinfty}
Assume the hypotheses of Theorem~\ref{thm:Laca-12-KMSbeta}, and suppose $N(s)>1$ for $s>e$.
If $\omega$ is a KMS$_\infty$ state of $\NT(X)$, then $\omega$ restricts to a tracial
state of $\Ff$.
\end{cor}


\section{From traces on $A$ to KMS$_\beta$ states on $\NT(X)$}\label{section:fromAtoNTX}

 In the first part of this section, we are going to simplify the characterization of KMS$_\beta$ states on $\NT(X)$ given by
the requirements in Theorem~\ref{thm:Laca-12-KMSbeta} in two steps. Firstly, we push the scaling
condition \eqref{scaling-converse} down from $\Ff$ to a scaling condition on elements in $i_e(A)=B_e$.
Secondly, under a suitable condition on the $\mathfrak{m}_{s,r}$, we show that the tracial property of $\phi$
on  $\Ff$  is automatically fulfilled as long as it  holds on $B_e$. The second part of this section contains
our construction of ground states and KMS$_\beta$ states from tracial states on the coefficient algebra $A$,
valid above a certain critical value of $\beta$.


\subsection{A simplification of the scaling condition for KMS$_\beta$ states.}
\begin{lemma}\label{restrictiontoA}
Let $X$ be a compactly aligned product system over $P$ with the coefficient algebra $A$. Let $\omega$
be a state on $\NT(X)$ and $\phi$ be its restriction to the core subalgebra $\Ff$.

(a) If $\omega$ is a KMS$_\beta$ state ($\beta>0$) for the dynamics $\sigma^N$ arising from
an injective homomorphism $N$ as in (\ref{def:sigmaN}),
then for $\xi\in X_s$, $\eta\in X_r$,  we have
\begin{equation}\label{phixieta}
\omega(i_s(\xi)i_r(\eta)^*) = \delta_{s,r}N(s)^{-\beta}\omega(i_e(\langle \eta,\xi\rangle_s)).
\end{equation}
for $s,r\in P$. Thus the restriction map $\omega\mapsto\omega\restriction_{i_e(A)}$ from the
KMS$_\beta$ states on $\NT(X)$ to traces on $i_e(A)$ is injective.

(b) In addition, assume that $X$ is of finite type. Then $\omega$ is a KMS$_\beta$ state ($\beta>0$)
for the dynamics $\sigma^N$ if and only if $\phi$ is a tracial state on $\Ff$ such that
$\omega=\phi\circ\Phi^\delta$ and
\begin{equation} \label{fintypeKMScond}
\phi(i_s(\31_j^s)i_e(a)i_s(\31_j^s)^*)=N(s)^{-\beta}\phi(i_e(a))
\end{equation}
for all $s\in P$, $j=0,\ldots,N_s-1$, and $a\in A$.
\end{lemma}
\begin{proof}
(a) Equality (\ref{phixieta}) follows immediately from Proposition~\ref{prop:KMS-state-on-F}.
It shows that a KMS$_\beta$ state is uniquely determined by its values on the coefficient algebra $A$.

(b) Assume $\phi$ is a tracial state on $\Ff$ which satisfies \eqref{fintypeKMScond}. We will show that
$\phi$ satisfies \eqref{scaling-converse} for all $y\in \Ff$. First, for $s\in P$,  $i, m\in\{0,\ldots,N_s-1\}$ and
$y\in \Ff$ we have
$$ \phi(i_s(\31_i^s)yi_s(\31_m^s)^*) = \phi(i_s(\31_m^s)i_s(\31_m^s)^*i_s(\31_i^s)yi_s(\31_m^s)^*) =0 $$
if $i\neq m$, by Definition \ref{def:finite-type} (3). Thus for
$y=i_r(\31_j^r)i_e(a)(i_r(\31_k^r) i_e(b))^*$ in $\Ff$,  where $j,k\in\{0,\ldots,N_r-1\}$ and $a,b\in A$, we have
\begin{align}
 \phi(i_s(\31_i^s) y i_s(\31_m^s)^*)
 &= \delta_{i,m}\phi(i_s(\31_i^s)i_r(\31_j^r) i_e(ab^*)
(i_s(\31_i^s)i_r(\31_k^r))^*)\notag\\
&=\delta_{i,m} \phi(i_{sr}(\31^{sr}_{i\cdot j})i_e(ab^*)i_{sr}(\31^{sr}_{i\cdot k})^*)\notag\\
&= \delta_{i,m}\delta_{i\cdot j, i\cdot k}N(sr)^{-\beta}\phi(i_e(ab^*))\;\;\;
\text{ by }\eqref{fintypeKMScond}\notag\\
&=\delta_{i,m}\delta_{j, k}N(sr)^{-\beta}\phi(i_e(ab^*))\notag\\
&=\delta_{i,m}N(s)^{-\beta}\phi(i_r(\31^r_j)i_e(ab^*)i_r(\31^r_k)^*)\;\;\;
 \text{ by }\eqref{fintypeKMScond}\notag\\
& = \delta_{i,m}N(s)^{-\beta}\phi(y), \notag
\end{align}
which proves \eqref{scaling-converse} for all such $y$. Since
an arbitrary spanning element $y$ in $\Ff$ is a linear combination of elements
$i_r(\31_j^r)i_e(a)(i_r(\31_k^r) i_e(b))^*$ by  Definition \ref{def:finite-type}(2), the scaling condition
\eqref{scaling-converse} is valid for all $y\in \Ff$.
Theorem~\ref{thm:Laca-12-KMSbeta} implies therefore that $\omega$ is a KMS$_\beta$ state
of $\NT(X)$.

The reverse implication is an immediate consequence of Theorem \ref{thm:Laca-12-KMSbeta}.
\end{proof}

%
%
%

\begin{lemma}\label{lem:scaling-trace-on-A}
Let $X$ be a product system of finite type over $P$ with the coefficient algebra $A$.
Let $\phi$ be a functional on $\NT(X)$. If $\phi$ satisfies
\begin{equation}\label{fintypeKMScond-left-action}
\phi(i_e(c)i_s(\31_j^s)i_s(\31^s_k)^*i_e(d)^*)=N(s)^{-\beta}\phi(i_e(\langle \varphi_s(d)\31^s_k,
\varphi_s(c)\31^s_j\rangle))
\end{equation}
for all $c,d\in A$, $s\in P$ and $j,k\in \{0, \dots ,N_{s}-1\}$, then $\phi$ satisfies
\begin{equation}\label{fintypeKMScond-prime}
\phi(i_s(\31_j^s)i_e(a)i_s(\31_l^s)^*)=\delta_{j,l}N(s)^{-\beta}\phi(i_e(a))
\end{equation}
for all $s\in P$, $j,l=0,\ldots,N_s-1$ and $a\in A$.

Conversely, if $\phi$ satisfies \eqref{fintypeKMScond-prime} and $\phi\circ i_e$ is a trace on $A$,
then $\phi$ satisfies \eqref{fintypeKMScond-left-action}.
\end{lemma}

\begin{proof} Assume that identity \eqref{fintypeKMScond-left-action} holds. For $a\in A$, $s\in P$ and
$j, l=0, \dots ,N_s-1$, write $\rho_s(a)\31^s_j=
\sum_{\mu=0}^{N_s-1}\varphi_s(a_\mu)\31^s_\mu$. Then
\begin{align*}
\phi(i_s(\31_j^s)i_e(a)i_s(\31_l^s)^*)
&=N(s)^{-\beta}\sum_{\mu}\phi(i_e(\langle \31^s_l, \varphi_s(a_\mu)\31^s_\mu\rangle))\\
&=N(s)^{-\beta}\phi(i_e(\langle \31^s_l, \rho_s(a)\31^s_j\rangle))\\
&=\delta_{j,l}N(s)^{-\beta}\phi(i_e(a)),
\end{align*}
as claimed in \eqref{fintypeKMScond-prime}.
Now suppose $\phi$ restricts to a trace on $i_e(A)$ and satisfies \eqref{fintypeKMScond-prime}. Let
$y=i_s(\varphi_s(c)\31^s_j)i_s(\varphi_s(d)\31^s_k)^*\in \Ff$, and
express $\varphi_s(c)\31^s_j=\sum_{\nu=0}^{N_s-1}\rho_s(c_\nu)\31^s_\nu$  and $\varphi_s(d)\31^s_k=\sum_{i=0}^{N_s-1}\rho_s(d_i)\31^s_i$.  Then
\begin{align*}
\phi(y)
&=\sum_\nu\sum_i\phi\bigl(i_s(\31_\nu^s)i_e(c_\nu(d_i)^*)i_s(\31^s_i)^*\bigr)\\
&=N(s)^{-\beta}\sum_{\nu,i}\delta_{\nu,i}\phi \bigl(i_e(c_\nu(d_i)^*)\bigr)\\
&=N(s)^{-\beta}\phi(\sum_\nu(i_e((d_\nu)^*c_\nu))),
\end{align*}
because $\phi\circ i_e$  is a trace.  This last term is, by the choice of $c_\nu$ and $d_\nu$, equal
to $N(s)^{-\beta}\phi(i_e(\langle \varphi_s(d)\31^s_k, \varphi_s(c)\31^s_j\rangle))$,
giving \eqref{fintypeKMScond-left-action}.
\end{proof}

\begin{rmk}\label{rmk:like-eq-8-2-LR}
\rm Under the hypothesis of part (b) of Lemma \ref{restrictiontoA}, the condition
$$
\phi(i_e(a)i_s(\31_j^s)i_s(\31_k^s)^*i_e(b^*)) =N(s)^{-\beta}\phi(i_e(\langle\31_k^s,\varphi_s(b^*a)\31_j^s\rangle)
$$
for all $s\in P$, $j,k\in\{0,\ldots,N_s-1\}$, $a,b\in A$,  is similar to \cite[equation (8.2)]{Lac-Rae2} in the
case of the product system over $\NN^\times$ from Example~\ref{ex-affine-toeplitz}.
\end{rmk}

\medskip
Given a tracial state on $i_e(A)$ satisfying the scaling identity (\ref{fintypeKMScond-prime}), formula (\ref{phixieta}) may
be used to extend it to a state on $\NT(X)$. The question remains if such a state is tracial on $\Ff$.
We examine this issue next. The following definition involves functions $\mathfrak{m}_{s,r}$ introduced in
Definition \ref{def:finite-type}.

\begin{dfn} Let $X$ be  an associative product system of finite type over $P$. We say the functions
$\mathfrak{m}_{s,r}$ \emph{respect co-prime pairs} if the following condition holds: for all $s,r\in P$ such that
$s\wedge r=e$, all $j,g,h\in \{0, \dots, N_s-1\}$ and all $l,m,n\in \{0, \dots ,N_r-1\}$ we have
\begin{equation}\label{cond-m}
\mathfrak{m}_{s,r}(j,m) =\mathfrak{m}_{r,s}(l, g)\text{ and }\mathfrak{m}_{s,r}(j,n) =
\mathfrak{m}_{r,s}(l, h) \;\; \Rightarrow \;\; m=n\text{ and }g=h.
\end{equation}
\end{dfn}

\begin{rmk}
\rm{For the product systems from Examples~\ref{ex-affine-toeplitz} and \ref{ex-additive-toeplitz} the maps $\mathfrak{m}_{s,r}$
respect co-prime pairs in the sense of \eqref{cond-m}. We only show this in the case of Example~\ref{ex-affine-toeplitz} because the argument is similar for the second example. We have $P=\NN^\times$ and $N_m=m$ for $m\in \NN^\times$.
Suppose that $s,r$ are co-prime integers. Since the product of $X_s$ with $X_r$ is implemented by
the endomorphism that raises the generating isometry $S$ to the power $s$, it follows that
$\31^{sr}_{j\cdot k}=\31^{sr}_{j+sk}$ for all $j\in \{0, \dots, N_s-1\}$ and $k\in \{0, \dots ,N_r-1\}$.
Assume $\mathfrak{m}_{s,r}(j,m) =\mathfrak{m}_{r,s}(l, g)$ and $\mathfrak{m}_{s,r}(j,n) =
\mathfrak{m}_{r,s}(l, h)$, where $j,g,h\in \{0, \dots, N_s-1\}$ and $l,m,n\in \{0, \dots ,N_r-1\}$. Then
$j-l=rh-sn=rg-sm$, and therefore $r(h-g)=s(n-m)$. Then necessarily $h=g$ and $n=m$, as required.
}
\end{rmk}

\begin{theorem}\label{thm:from-trace-A-to-trace-F}
Let $X$ be an associative product system of finite type over $P$ such that $\mathfrak{m}_{s,r}$
are bijective for all $s,r\in P$ and respect co-prime pairs.

 If $\phi$ is a state of $\Ff$ such that  $\phi\upharpoonright_{i_e(A)}$ is
 a trace and for some $0<\beta <\infty$ we have
 \begin{equation}\label{fintypeKMScond-jl}
 \phi(i_s(\31^s_j)i_e(a)i_s(\31^s_l)^*)=\delta_{j,l}N(s)^{-\beta}\phi(i_e(a))
 \end{equation}
 for all $a\in A$, $s\in P$, $j,l=0,\dots ,N_s-1$, then $\phi$ is a trace on $\Ff$. In particular, $\phi\circ \Phi^\delta$ is a KMS$_\beta$-state of $\NT(X)$.
\end{theorem}

Before giving the proof of this theorem we need some preparation.  First we introduce some notation. For $s,r$ in $P$, if $\mathfrak{m}_{s,r}$ is bijective, then every $k=0,\dots , N_{sr}-1$ has a unique decomposition $k=k(s)\cdot k(r)$ in
$\{0,\dots, N_{s}-1\}\times \{0, \dots , N_{r}-1\}$, and we must have
$$
N_{sr}=N_sN_r.
$$
For  $s,r$ in $P$ let
$$
s'=r^{-1}(s\vee r)\text{ and }r'=s^{-1}(s\vee r),
$$
and note then that $s\vee r=sr'=rs'$ as well as $(s\wedge r)s'=s$ and $(s\wedge r)r'=r$, with
$s\wedge r$ denoting the greatest lower bound  of $s$ and $r$.

As noticed in \cite{HLS}, if the
product system $X$ is such that $I_s\in \Kk(X_s)$ for all $s\in P$, then \cite[Proposition 5.10]{Fow} proves the
stronger statement that every product of the form $i_s(\xi)^*i_r(\eta)$  in $\NT(X)$ is a linear
combination (rather than a limit of linear combinations) of elements of the form
$i_e(a)i_{r'}(\xi')i_{s'}(\eta')^*i_e(b)$ for appropriate $a,b\in A$ and $\xi'\in X_{r'}$,
$\eta'\in X_{s'}$.  The next result makes this
decomposition explicit in the case of $X$ of finite type with bijective maps counting
the elements in the bases.

\begin{lemma}
Let $X$ be a product system of finite type over $P$ such that $\mathfrak{m}_{s,r}$ is bijective, for all $s,r\in P$.
For any $\xi\in X_s$ and $\eta\in X_r$, where $s,r\in P$, we have
\begin{equation}\label{Neals-precise-formula}
i_s(\xi)^*i_r(\eta)=\sum_{i=0}^{N_{s\vee r}-1}i_e(\langle \xi, \31^s_{i(s)}\rangle_s)i_{r'}(\31^{r'}_{i(r')})i_{s'}(\31^{s'}_{i(s')})^*i_e(\langle \eta, \31^r_{i(r)}\rangle_r)^*.
\end{equation}
\end{lemma}

\begin{proof} Writing $i_s(\xi)^*i_r(\eta)=i_s(\xi)^*i^{(s\vee r)}(I_{s\vee r})i_r(\eta)$, and using \eqref{eq:Is-compact} and the properties of the
multiplication in $X$ gives \eqref{Neals-precise-formula}.
\end{proof}

\begin{proof}[Proof of Theorem~\ref{thm:from-trace-A-to-trace-F}.] Let $\phi$ be a state of $\Ff$ such that
$\phi\circ i_e$ is a tracial state and \eqref{fintypeKMScond-jl} is satisfied.
Let $y_1=i_s(\31^s_j)i_e(ab^*)i_s(\31^s_k)^*$ and $y_2=i_r(\31^r_m)i_e(cd^*)i_r(\31^r_n)^*$ be spanning elements in
$\Ff$, where $a,b,c,d \in A$, $s,r\in P$, $j,k\in \{0, \dots ,N_s-1\}$ and $m,n\in \{0, \dots ,N_r-1\}$. To prove that $\phi$ is a trace on $\Ff$, it suffices to show that
\begin{equation}\label{phi-trace}
\phi(y_1y_2)=\phi(y_2y_1).
\end{equation}
Using \eqref{Neals-precise-formula}, we have
\begin{align*}
i_s(\rho_s(ba^*)&\31^s_k)^*i_r(\rho_r(cd^*)\31^r_m)\\
&=\sum_{i=0}^{N_{s\vee r}-1}i_e(\langle \rho_s(ba^*)\31^s_k, \31^s_{i(s)}\rangle_s)i_{r'}(\31^{r'}_{i(r')})i_{s'}(\31^{s'}_{i(s')})^*
i_e(\langle \rho_r(cd^*)\31^r_m , \31^r_{i(r)}\rangle_r)^*\\
&=\sum_{i=0}^{N_{s\vee r}-1}i_e(ab^*\langle\31^s_k, \31^s_{i(s)}\rangle_s)i_{r'}(\31^{r'}_{i(r')})i_{s'}(\31^{s'}_{i(s')})^*
i_e(dc^*\langle \31^r_m , \31^r_{i(r)}\rangle_r)^*\\
&=i_e(ab^*)i_{r'}(\31^{r'}_{k'})i_{s'}(\31^{s'}_{m'})^*
i_e(dc^*)^*  \;\;\; \text{ by Definition }~\ref{def:finite-type}\; (3),
\end{align*}
where $k'$ is the unique element in $\{0, \dots ,N_{r'}-1\}$ and $m'$ the unique element in
$\{0, \dots ,N_{s'}-1\}$ such that
$k\cdot k'=m\cdot m'$ in $\{0, \dots ,N_{s\vee r}-1\}$. It follows that
$$
y_1y_2=i_s(\31^s_j) i_e(ab^*)i_{r'}(\31^{r'}_{k'})i_{s'}(\31^{s'}_{m'})^*
i_e(dc^*)^*(i_r(\31^r_n))^*.
$$
Now invoking Definition~\ref{def:finite-type} (2) we can write
$\varphi_{r'}(ab^*)\31^{r'}_{k'}=\sum_{h=0}^{N_{r'}-1}\rho_{r'}(e_h)\31^{r'}_h$
and
$\varphi_{s'}(dc^*)\31^{s'}_{m'}=\sum_{i=0}^{N_{s'}-1}\rho_{s'}(f_i)\31^{s'}_i$.
Then by regrouping terms in $y_1y_2$ we have
\begin{align}
\phi(y_1y_2)
&=\sum_{h=0}^{N_{r'}-1}\sum_{i=0}^{N_{s'}-1}\phi\bigl(i_{s\vee r}(\31^{s\vee r}_{j\cdot h})i_e(e_hf_i^*)i_{s\vee r}(\31^{s\vee r}_{n\cdot i})^*\bigr)\notag \\
&=N(s\vee r)^{-\beta}\sum_{h=0}^{N_{r'}-1}\sum_{i=0}^{N_{s'}-1}\delta_{j\cdot h, n\cdot i}\phi(i_e(e_hf_i^*))
\;\;\; \text{ by } \eqref{fintypeKMScond-jl}\notag\\
&=N(s\vee r)^{-\beta}\phi\bigl(i_e(\langle \31^{r'}_h, \varphi_{r'}(ab^*)\31^{r'}_{k'}\rangle\langle \31^{s'}_{m'}, \varphi_{s'}(cd^*)\31^{s'}_i\rangle)\bigr),\label{product-ip1}
\end{align}
where $h\in \{0, \dots ,N_{r'}-1\}$ and $i\in \{0, \dots, N_{s'}-1\}$ are uniquely determined such that
$\mathfrak{m}_{s,r'}(j, h)= \mathfrak{m}_{r,s'}(n, i)$ in $\{0, \dots ,N_{s\vee r}-1\}$.

Also, invoking  Definition~\ref{def:finite-type}(1) we  write
$\rho_s(ab^*)\31^s_j=\sum_{g=0}^{N_s-1}\varphi_s(u_g)\31^s_g$ and $\rho_r(dc^*)\31^r_n=\sum_{h=0}^{N_r-1}\varphi_r(w_h)\31^r_h$.
Hence $y_1y_2=\sum_g\sum_hi_e(u_g)i_{s\vee r}(\31^{s\vee r}_{g\cdot k'})i_{s\vee r}(\31^{s\vee r}_{h\cdot m'})^*i_e(w_h)^*$. Since $\phi\circ i_e$ is a trace, $\phi$ satisfies \eqref{fintypeKMScond-left-action}, and so
\begin{align}
\phi(y_1y_2)
&=\sum_g\sum_h N(s\vee r)^{-\beta}\phi(i_e(\langle \varphi_{s\vee r}(w_h) \31^{s\vee r}_{h\cdot m'},
\varphi_{s\vee r}(u_g)\31^{s\vee r}_{g\cdot k'}\rangle))\notag\\
&=N(s\vee r)^{-\beta}\phi(i_e(\langle F^{r,s'}(\31^{r}_{n}\otimes_A \varphi_{s'}(dc^*)\31^{s'}_{m'}),
F^{s,r'}(\31^{s}_{j}\otimes_A \varphi_{r'}(ab^*)\31^{r'}_{k'})\rangle)).\label{F-s-r}
\end{align}
Write  $\31^s_j=F^{s\wedge r,s'}({\31^{s\wedge r}_{j(s\wedge r)}\otimes \31^{s'}_{j(s')}})$ and $\31^r_n=
F^{s\wedge r,r'}(\31^{s\wedge r}_{n(s\wedge r)}\otimes \31^{r'}_{n(r')})$. The associativity yields a
decomposition
$$
F^{(s\wedge r)r', s'}(F^{s\wedge r, r'}\otimes_A {I_{s'}})= F^{s\wedge r, r's'}(I_{s\wedge r}\otimes_A F^{r',s'}).
$$
Now using the definition of the inner product in $X_{s'r'}=X_{r's'}$, the term under $\phi$ in \eqref{F-s-r} is seen
to be equal to
\begin{align*}
\phi(&i_e(\langle F^{r,s'}(\31^{r}_{n}\otimes_A \varphi_{s'}(dc^*)\31^{s'}_{m'}),
F^{s,r'}(\31^{s}_{j}\otimes_a \varphi_{r'}(ab^*)\31^{r'}_{k'})\rangle))\\
&\hspace{-2mm}=\phi\bigl(i_e(\langle F^{r', s'}(\31^{r'}_{n(r')}\otimes\varphi_{s'}(dc^*)\31^{s'}_{m'}), \varphi_{s'r'}(\langle \31^{s\wedge r}_{n(s\wedge r)}, \31^{s\wedge r}_{j(s\wedge r)}\rangle)F^{s', r'}(\31^{s'}_{j(s')}\otimes
\varphi_{r'}(ab^*)\31^{r'}_{k'} )\rangle)\bigr).
\end{align*}
Thus Definition~\ref{def:finite-type}(3) implies that  $\phi(y_1y_2)=0$ unless the equality
$n(s\wedge r)=j(s\wedge r)$ holds, in
which case
\begin{equation}\label{eq:otherproduct-ip1}
\phi(y_1y_2)=N(s\vee r)^{-\beta}\phi\bigl(i_e (\langle F^{r', s'}(\31^{r'}_{n(r')}\otimes\varphi_{s'}(dc^*)\31^{s'}_{m'}), F^{s', r'}(\31^{s'}_{j(s')}\otimes
\varphi_{r'}(ab^*)\31^{r'}_{k'} )\rangle ) \bigr).
\end{equation}

Similarly, if we let $n',g\in \{0,\dots ,N_{s'}-1\}$ and $j',l\in \{0, \dots ,N_{r'}-1\}$ be the uniquely determined
elements such that $
n\cdot n'=j\cdot j'\text{ and }m\cdot g=k\cdot l$ in $\{0, \dots ,N_{s\vee r}-1\}$, then by employing
Definition~\ref{def:finite-type}(2) we have
$$
\phi(y_2y_1)
=N(s\vee r)^{-\beta}\phi\bigl(i_e( \langle \31^{s'}_g, \varphi_{s'}(cd^*)\31^{s'}_{n'}\rangle_{s'}\langle \31^{r'}_{j'}, \varphi_{r'}(ab^*)\31^{r'}_l\rangle_{r'})\bigr).
$$
Since $\phi\circ i_e$ is a trace on $A$, we can rewrite this as
\begin{equation}\label{product-ip2}
\phi(y_2y_1)=N(s\vee r)^{-\beta}\phi\bigl(i_e( \langle \31^{r'}_{j'}, \varphi_{r'}(ab^*)\31^{r'}_l\rangle_{r'}\langle \31^{s'}_g, \varphi_{s'}(cd^*)\31^{s'}_{n'}\rangle_{s'})\bigr)
\end{equation}
On the other hand, by employing Definition~\ref{def:finite-type}(1) we obtain
$$
\phi(y_2y_1)=N(s\vee r)^{-\beta}\phi\bigl(i_e (\langle F^{s', r'}(\31^{s'}_{k(s')}\otimes
\varphi_{r'}(dc^*)\31^{r'}_{j'}), F^{r', s'}(\31^{r'}_{m(r')}\otimes
\varphi_{s'}(ab^*)\31^{s'}_{n'} )\rangle ) \bigr)
$$
if $k(s\wedge r)=m(s\wedge r)$, and $\phi(y_2y_1)=0$ otherwise.

\vskip 0.2cm
{\bf{Case 1.}} $k(s\wedge r)=m(s\wedge r)$ and $n(s\wedge r)=j(s\wedge r)$. Then the equalities
$k\cdot k'=m\cdot m'$ and $j\cdot j'=n\cdot n'$ yield
\begin{align*}
k(s\wedge r) \cdot k(s')\cdot k'&= m(s\wedge r)\cdot m(r')\cdot m'\text{ and}\\
j(s\wedge r)\cdot j(s')\cdot j'&= n(s\wedge r)\cdot n(r')\cdot n',
\end{align*}
and so uniqueness of decomposition in $\{0,\dots, N_{s\wedge r}-1\}\times\{0, \dots, N_{s'r'}-1\}$ implies that
\begin{align*}
\mathfrak{m}_{s',r'}(k(s'), k') &=\mathfrak{m}_{r',s'}(m(r'), m')\text{ and}\\
\mathfrak{m}_{s',r'}(j(s'), j') &=\mathfrak{m}_{r',s'}(n(r'), n').
\end{align*}
For the same reason,  $k\cdot l=m\cdot g$ and $j\cdot h=n\cdot i$ imply that
\begin{align*}
\mathfrak{m}_{s',r'}(k(s'), l) &= \mathfrak{m}_{r',s'}(m(r'),g)\text{ and}\\
\mathfrak{m}_{s',r'}(j(s'), h) &= \mathfrak{m}_{r',s'}(n(r'), i)
\end{align*}
Since $s'\wedge r'=e$, the assumption that  $\mathfrak{m}_{s',r'}$ respects co-prime pairs implies that
\begin{equation}\label{all-indices-equal}
k'=l,\, m'=g, \,
j'=h,\,\text{ and } n'=i.
\end{equation}

Hence, by \eqref{product-ip1} and \eqref{product-ip2}, the expressions for $\phi(y_1y_2)$ and $\phi(y_2y_1)$ become
\begin{align*}
\phi(y_1y_2)&=N(s\vee r)^{-\beta}\phi\bigl(i_e(\langle \31^{r'}_{j'},
\varphi_{r'}(ab^*)\31^{r'}_{k'}\rangle_{r'}\langle \31^{s'}_{m'}, \varphi_{s'}(cd^*)\31^{s'}_{n'}\rangle_{s'})\bigr), \\
\phi(y_2y_1)&=N(s\vee r)^{-\beta}\phi\bigl(i_e( \langle \31^{r'}_{j'}, \varphi_{r'}(ab^*)\31^{r'}_{k'}\rangle_{r'}\langle \31^{s'}_{m'}, \varphi_{s'}(cd^*)\31^{s'}_{n'}\rangle_{s'})\bigr),
\end{align*}
establishing $\phi(y_1y_2)=\phi(y_2y_1)$ in this case.

\vskip 0.2cm
{\bf{Case 2.}} $m(s\wedge r)\neq k(s\wedge r)$. In this case we already saw that $\phi(y_2y_1)=0$. However,
\begin{align*}
i_s(\31^s_k)^*i_r(\31^r_m)
&=i_{s'}(\31^{s'}_{k(s')})^*i_e(\langle \31^{s\wedge r}_{k(s\wedge r)}, \31^{s\wedge r}_{m(s\wedge r)}\rangle)
i_{r'}(\31^{r'}_{m(r')})\\
&=\delta_{k(s\wedge r), m(s\wedge r)}i_{s'}(\31^{s'}_{k(s')})^*i_{r'}(\31^{r'}_{m(r')})=0,
\end{align*}
by Definition~\ref{def:finite-type}(3). Hence $y_1y_2=0$ and thus $\phi(y_1y_2)=\phi(y_2y_1)$.

\vskip 0.2cm
{\bf{Case 3.}} $j(s\wedge r)\neq n(s\wedge r)$. Similarly to case 2, we have $\phi(y_1y_2)=0$ by previous consideration, and
$y_2y_1=0$ by Definition~\ref{def:finite-type}(3), so again $\phi(y_1y_2)=\phi(y_2y_1)$.

\vskip 0.2cm
{\bf{Case 4.}} Finally, if $m(s\wedge r)\neq k(s\wedge r)$ and $j(s\wedge r)\neq n(s\wedge r)$
then $y_1y_2=0=y_2y_1$.
\end{proof}

Note that if $\phi\circ i_e$ is injective, then \eqref{product-ip1} and \eqref{eq:otherproduct-ip1} show
that the product system $X$ must satisfy the condition
\begin{equation}
\langle F^{s,r}(\31^s_j\otimes_A \varphi_r(a)\31^r_m), F^{r,s}(\31^r_n\otimes_A \varphi_s(b)\31^s_k)\rangle_{sr}
=\langle \varphi_r(a)\31^r_m,\31^r_n\rangle_r\langle \31^s_j, \varphi_s(b)\31^s_k\rangle_s
\end{equation}
for all $a,b\in A$, all $s,r\in P$ such that $s\wedge r=e$ and all $j,k=0,\dots, N_s-1$, $m,n=0, \dots , N_r-1$.


\subsection{Ground states and KMS$_\beta$ states of $\NT(X)$ induced from  states of $A$. }

We begin this section by recalling the construction of the induced representation via a right Hilbert
module. We refer to \cite{RW} for details. Let $Y$ be a right Hilbert $A$-module and  assume that $\varphi:B
 \to \mathcal{L}(Y)$ is a $*$-homomorphism. Suppose that $\pi:A\to B(H_\pi)$ is a representation. The
 balanced tensor product space $Y\otimes_A H_\pi$ is a Hilbert space where the inner-product is characterised
 by
 \begin{equation}\label{tensorinnerproduct}
 \langle\xi\otimes_A h, \eta\otimes_A k\rangle=(\pi(\langle \eta,\xi\rangle)h\mid k)
\end{equation}
for $\xi,\eta\in Y$ and $h,k\in H_\pi$. The induced representation $\Ind\pi$ of $B$ on  $Y\otimes_A H_\pi$
acts by
\begin{equation}\label{eq_formula-ind-repr}
\Ind\pi (b) (\xi\otimes_A h)=(\varphi(b)\xi) \otimes_A h.
\end{equation}
We apply this construction to the Fock module $F(X)$ associated to a product system $X$ over $P$ of right Hilbert
$A$--$A$-bimodules, see section~\ref{Prelim-3}. For a compactly aligned product system $X$, the Fock
representation $l$ of $X$ in $\mathcal{L}(F(X))$ gives rise to a $*$-homomorphism $l_*:\NT(X) \to \mathcal{L}(F(X))$.

\begin{rmk}
 \rm Since the left action has image in $\Kk(X_s)$ for every $s\in P$, \cite[Theorem 6.3]{Fow} says that $\NT(X)$ is isomorphic to a certain crossed-product $B_P\rtimes_{\tau, X}P$ (the proof there uses that every $X_s$ is essential, but applies in our setting due to Definition~\ref{def:finite-type}(1)).  Then the remark following \cite[Definition 7.1]{Fow}  indicates that
$B_P\rtimes_{\tau, X}P$, which is a universal crossed product for an action of $P$ on $B_P$ twisted by $X$, is isomorphic to
the associated reduced crossed product. We infer from this that $l_*$ is faithful.
\end{rmk}

Given a state $\tau$ on $A$, let $(\pi_\tau, h_\tau, H_\tau)$ be the corresponding GNS-representation. We denote $\31=\31_e\oplus \bigoplus_{s\neq e}0_s$ in $F(X)$. Consider the representation
$$
\Ind\pi_\tau : \NT(X)\to B(F(X)\otimes_A H_\tau)),
$$
and let
\begin{equation}\label{omegatautilde}
\tilde{\omega}_\tau(y)=(\Ind\pi_\tau(y) (\31\otimes h_\tau) \mid \31\otimes h_\tau)
\end{equation}
be the state of $\NT(X)$ arising from this representation. We claim that
\begin{equation}\label{eq:tilde-omega-GNS}
\tilde{\omega}_\tau(y)=
\begin{cases}
\tau(ab^*)&\text{ if } s=r=e\\
0&\text{ otherwise }
\end{cases}
\end{equation}
for $y=i_e(a)i_s(\31^s_k)i_r(\31^r_m)^*i_e(b)^*\in \NT(X)$.
It follows from identities (\ref{tensorinnerproduct}) and \eqref{eq_formula-ind-repr} that
$\tilde{\omega}_\tau(y)=\langle (l_*(y)\31)\otimes h_\tau, \31\otimes h_\tau\rangle$. The
characterization of $l_*$  shows that $l_r(\varphi_r(b)\31^r_m)^*\31=0_r$
unless $r=e$, which in turn implies $l_*(y)\31=0$ when $r\neq e$. Assuming $r=e$, we see next that
$l_*(y)\31$, as an element in $F(X)$, has a non-zero coordinate only at $s$, where it equals $\rho_s(b^*)(\varphi_s(a)\31^s_k)$.
To compute further in $\tilde{\omega}_\tau(y)$, note that the characterization of the inner-product on $F(X)\otimes_A H_\tau$
involves computing the inner-product in $F(X)$ given by $\langle \31, l_*(y)\31\rangle$.
By the definition of the inner-product on
$F(X)$ it follows that a non-zero contribution in $\langle \31, l_*(y)\31\rangle$ is
only possible at $s=e$, where it equals
$\langle \31_e, i_e(ab^*)\31_e\rangle_e$. In other words, we have established that $r\neq e$ or $s\neq e$ imply
$\tilde{\omega}_\tau(y)=0$, while for $s=r=e$ we obtain $\tilde{\omega}_\tau(y)=\bigl(\pi_\tau(\langle \31_e, i_e(ab^*)\31_e\rangle)h_\tau \mid h_\tau\bigr)$,
which means $\tilde{\omega}_\tau(y)=\tau(ab^*)$, as required in \eqref{eq:tilde-omega-GNS}. Thus, in connection with
Theorem~\ref{thm:Laca-12-ground}, we have the following result.

\begin{prop}\label{prop:ground-states}
For each state $\tau$ of $A$, the induced state $\tilde{\omega}_\tau$ given by (\ref{omegatautilde})
is a ground state of $\NT(X)$. Moreover, the assignment $\tau\mapsto \tilde{\omega}_\tau$ is
an affine isomorphism.
\end{prop}

Next we investigate if a similar construction can induce KMS$_\beta$ states of $\NT(X)$.
For each $s\in P$, let $\bar{\31}^s_j$ be the vector in $F(X)$ with component equal to $\31^s_j$ at $s$ and
$0_r$ for $r\neq s$.
Note that $l_*(i_s(\31^s_j))\31=l_s(\31^s_j)\31=\bar{\31}^s_j$ for every $s\in P$. We
define a series with positive terms by
\begin{equation}\label{def:series}
\zeta_N(\beta):=\sum_{s\in P}N(s)^{-\beta} N_s.
\end{equation}

\begin{theorem}\label{thm:KMsbeta-from-tracetau} Let $(G, P)$ be lattice ordered and let $X$ be an
associative product system of finite type over $P$ such that $\mathfrak{m}_{s,r}$ are bijective and respect
co-prime pairs, for all $s,r\in P$.  Assume that the series in \eqref{def:series} is convergent in
an interval $(\beta_c,\infty)$ for some $\beta_c>0$.

Let $\beta>\beta_c$. Then for a tracial state $\tau $ of $A$ there is a state $\omega_\tau:\Ff\to\CC$ given by
\begin{equation}\label{def:state-from-induced}
\omega_\tau(y)= \sum_{\{s\in P:r\leq s\}}\frac{N(s)^{-\beta}}{\zeta_N(\beta)}\sum_{j=0}^{N_s-1}
\tau(\langle \varphi_{r^{-1}s}(\langle \xi,\31^r_{j(r)}\rangle)\31^{r^{-1}s}_{j'},
\varphi_{r^{-1}s}(\langle \eta,\31^r_{j(r)}\rangle)\31^{r^{-1}s}_{j'}\rangle)
\end{equation}
for $y=i^{(r)}(\theta_{\xi,\eta})\in B_r$, where for each $j$ in the summation we denote $j'=j(r^{-1}s)$. Further,
the assignment $\tau\mapsto \omega_\tau\circ \Phi^\delta$ is an affine continuous map from the set of tracial states of $A$ to a subset of KMS$_\beta$ states of $\NT(X)$.
\end{theorem}
Before giving the proof of this theorem we need some preparation.
To argue that \eqref{def:state-from-induced} defines a state of $\Ff$, we first show that $\omega_\tau$ can be
defined alternatively on $\Ff$ as
\begin{equation}\label{eq:sum-of-vector-states}
\omega_\tau(y)= \frac{1}{\zeta_N(\beta)} \sum_{s\in P}{N(s)^{-\beta}}\sum_{j=0}^{N_s-1}
\langle\Ind\pi_\tau(y) (l_s({\31}^s_j)\31\otimes h_\tau),\,\, (l_s({\31}^s_j)\31\otimes h_\tau)\rangle
\end{equation}
for $y\in B_r$, $r\in P$. As in the proof of \cite[Theorem 20]{Lac},  $\omega_\tau$ is an absolutely convergent
infinite linear combination of vector states in $\Ind\pi_\tau$, so it is
absolutely continuous with respect to $\tilde{\omega}_\tau$. Also, $\tilde{\omega}_\tau$ is absolutely
continuous with respect to $\omega_\tau$,
since the vectors $\{l_*(i_s({\31}^s_j))\31\otimes h_\tau\}_{s\in P}$
form a generating set for $\Ind\pi_\tau$.
Now, by definition of the Fock representation, $l_*(y)\bar{\31}^s_j=0$ when
$r\not\leq s$. Assume therefore $s\in rP$.
Since  $l_s({\31}^s_j)\31=\bar{\31}^s_j$, the summand in $s$ and $j$ in the right-hand side of
\eqref{eq:sum-of-vector-states} is
\begin{align}
\langle\Ind\pi_\tau(y) \bigl((l_s({\31}^s_j)\31)\otimes h_\tau\bigr),\,\, (l_s({\31}^s_j)\31)\otimes h_\tau\rangle
&=\langle\Ind\pi_\tau(y) (\bar{\31}^s_j\otimes h_\tau),\,\, \bar{\31}^s_j\otimes h_\tau\rangle\notag\\
&=\langle (l_*(y)\bar{\31}^s_j)\otimes h_\tau,\,\, \bar{\31}^s_j\otimes h_\tau\rangle\notag\\
&=\left(\pi_\tau(\langle \bar{\31}^s_j, l_*(y)\bar{\31}^s_j\rangle_{F(X)})h_\tau\mid h_\tau\right)\notag\\
&=\tau(\langle {\31}^s_j, i_r^s(\theta_{\xi,\eta}){\31}^s_j\rangle_s).\label{eq:Ind-pi-tau}
\end{align}
At this stage we put $j'=j(r^{-1}s)$, so that $j$ is given uniquely by $\mathfrak{m}_{r,r^{-1}s}(j(r), j')=j$. We
decompose $\31^s_j=F^{r, r^{-1}s}(\31^r_{j(r)}\otimes \31^{r^{-1}s}_{j'})$, and we use the
properties of the balanced inner product on $X_r\otimes_A X_{r^{-1}s}$ to write further
\begin{align*}
\langle\Ind\pi_\tau(y) &(l({\31}^s_j)\31\otimes h_\tau),\,\, (l({\31}^s_j)\31)\otimes h_\tau\rangle\\
&=\tau(\langle \31^{r^{-1}s}_{j'}, \varphi_{r^{-1}s}(\langle\31^r_{j(r)},\xi\rangle\langle \eta, \31^r_{j(r)}\rangle)\31^{r^{-1}s}_{j'}\rangle)\\
&=\tau(\langle \varphi_{r^{-1}s}(\langle \xi, \31^r_{j(r)}\rangle)\31^{r^{-1}s}_{j'}, \varphi_{r^{-1}s}(\langle \eta, \31^r_{j(r)}\rangle)\31^{r^{-1}s}_{j'}\rangle).
\end{align*}
The last term is the summand under $s$ and $j$ in \eqref{def:state-from-induced}. This shows that formulae
(\ref{def:state-from-induced}) and (\ref{eq:sum-of-vector-states}) yield the same positive functional.
Finally,  functional $\omega_\tau$ is a state because
\begin{align*}
\omega_\tau(1)
&=\sum_{s\in P} \frac{N(s)^{-\beta}}{\zeta_N(\beta)} \sum_{j=0}^{N_s-1}\tau(\langle \31^s_{j(s)},\31^s_{j(s)}\rangle)\\
&=\sum_{s\in P}\frac{N(s)^{-\beta}}{\zeta_N(\beta)}  \sum_{j=0}^{N_s-1}\tau(1)\\
&=(\zeta_N(\beta))^{-1}\sum_{s\in P}N(s)^{-\beta}N_s=1.
\end{align*}
Next we want to employ Theorem~\ref{thm:from-trace-A-to-trace-F} to
show that $\omega_\tau$ given by \eqref{def:state-from-induced} is a trace of $\Ff$. In the
next lemmas we verify  that the assumptions of Theorem~\ref{thm:from-trace-A-to-trace-F} are
fulfilled by $\omega_\tau$.

\begin{lemma}\label{lemma:1}
The map $\omega_\tau$ given by \eqref{def:state-from-induced} is a trace on
$B_e=i_e(A)$ whenever $\tau$ is a trace on $A$.
\end{lemma}

\begin{proof} Assume $\tau$ is a trace on $A$.  Then \eqref{def:state-from-induced} implies
\begin{equation}\label{eq:ometau-onA}
\omega_\tau(i_e(a))=\frac 1{\zeta_N(\beta)}\sum_{s\in P}N(s)^{-\beta}\sum_{j=0}^{N_s-1}
\tau(\langle \31^s_j,\varphi_s(a)\31^s_j\rangle)
\end{equation}
for each $a\in A$. Let $c,d\in A$. To prove that $\omega_\tau(i_e(cd))=\omega_\tau(i_e(dc))$ it suffices to show that
\begin{equation}\label{eq:omegatau-trace-condition}
\sum_{j=0}^{N_s-1}\tau(\langle \31^s_j,\varphi_s(cd)\31^s_j\rangle)=\sum_{j=0}^{N_s-1}\tau(\langle \31^s_j,\varphi_s(dc)\31^s_j\rangle)
\end{equation}
for $c,d\in A$.
Using Definition~\ref{def:finite-type}(2) we can write $\varphi_s(d)\31^s_j=\sum_{n=0}^{N_s-1}\rho_s(d_{n,j})\31^s_n$  and $\varphi_s(c)\31^s_l=\sum_{m=0}^{N_s-1}\rho_s(c_{m, n})\31^s_m$, for every $j,l=0, \dots ,N_s-1$.

Then the left-hand side of \eqref{eq:omegatau-trace-condition} can be written as follows
\begin{align}
\sum_{j=0}^{N_s-1}\tau(\langle \31^s_j,\varphi_s(cd)\31^s_j\rangle)
&=\sum_{j=0}^{N_s-1}\tau(\langle \31^s_j,\varphi_s(c)( \sum_{n=0}^{N_s-1}\rho_s(d_{n,j})\31^s_n )\rangle)\notag\\
&=\sum_{j=0}^{N_s-1}\tau(\langle \31^s_j,\sum_{n=0}^{N_s-1} \sum_{m=0}^{N_s-1}\rho_s(c_{m,n}d_{n,j})\31^s_m \rangle))\notag\\
&=\sum_{j=0}^{N_s-1}\sum_{n=0}^{N_s-1}\tau(c_{j,n}d_{n,j}) \;\;\;
\text{ by Definition~\ref{def:finite-type}(3)}\notag\\
&=\sum_{j=0}^{N_s-1}\sum_{n=0}^{N_s-1}\tau(d_{n,j}c_{j,n}) \;\;\;
\text{ since }\tau \text{ is a trace}.
\label{eq:cd}
\end{align}
Similarly, the right-hand side of \eqref{eq:omegatau-trace-condition} is
\begin{align}
\sum_{j=0}^{N_s-1}\tau(\langle \31^s_j,\varphi_s(dc)\31^s_j\rangle)
&=\sum_{j=0}^{N_s-1}\tau(\langle \31^s_j,\varphi_s(d)(\sum_{m=0}^{N_s-1}\rho_s(c_{m,j})\31^s_m )\rangle)\notag\\
&=\sum_{j=0}^{N_s-1}\tau(\langle \31^s_j,\sum_{m=0}^{N_s-1} \sum_{n=0}^{N_s-1}\rho_s(d_{n,m}c_{m,j})\31^s_n \rangle))\notag\\
&=\sum_{j=0}^{N_s-1}\sum_{m=0}^{N_s-1}\tau(d_{j,m}c_{m,j}) \;\;\;
\text{ by Definition~\ref{def:finite-type}(3)}\label{eq:dc}.
\end{align}
Thus \eqref{eq:cd} and \eqref{eq:dc} show that $\omega_\tau$ is a trace on $i_e(A)$.
\end{proof}

\begin{lemma}\label{lemma:2}
The state $\omega_\tau$ of $\Ff$ given by \eqref{def:state-from-induced} satisfies \eqref{fintypeKMScond-jl}.
\end{lemma}
\begin{proof}  Let $a\in A$ and $n,m\in \{0, \dots ,N_r-1\}$. When computing  $\omega_\tau(i_r(\31^r_n)i_e(a)i_r(\31^r_m)^*)$
using \eqref{def:state-from-induced}, we have $\xi=\31^r_n$ and $\eta=\rho_r(a^*)\31^r_m$.
Thus, in the summation
over $j=0, \dots ,N_s-1$, the terms $\langle \31^r_n, \31^r_{j(r)}\rangle$ are zero for every $s\geq r$  unless
$j(r)=n$, that is unless $j=n\cdot {j'}$ for $j'=0, \dots ,N_{r^{-1}s}-1$. Hence the summation
over $j$ is simply a summation
over $j'$. Moreover, since $\langle \rho_r(a^*)\31^r_m, \31^r_{j(r)}\rangle=a\langle \31^r_m, \31^r_{n}\rangle$, we
also get a zero contribution unless $n=m$. In other words, the left-hand side of \eqref{fintypeKMScond-jl} is
\begin{equation}\label{omegatau-left}
\omega_\tau(i_r(\31^r_n)
i_e(a)i_r(\31^r_m)^*)=\delta_{m,n} \sum_{\{s\in P: r\leq s\}} \frac {N(s)^{-\beta}}{\zeta_N(\beta)}
\sum_{j'=0}^{N_{r^{-1}s-1}}\tau( \langle\31^{r^{-1}s}_{j'}, \varphi_{r^{-1}s}(a)  \31^{r^{-1}s}_{j'}\rangle).
\end{equation}
Now applying \eqref{eq:ometau-onA} we can rewrite the right-hand side of \eqref{fintypeKMScond-jl} as follows
\begin{align*}
\delta_{n,m}N(r)^{-\beta}\omega_\tau(i_e(a))
&=\delta_{n,m} \frac 1{\zeta_N(\beta)}\sum_{q\in P}N(r)^{-\beta}N(q)^{-\beta}\sum_{l=0}^{N_q-1}
\tau(\langle \31^q_l,\varphi_q(a)\31^q_l\rangle)\\
&=\delta_{n,m} \frac 1{\zeta_N(\beta)}\sum_{q\in P}N(rq)^{-\beta}\sum_{l=0}^{N_q-1}
\tau(\langle \31^q_l,\varphi_q(a)\31^q_l\rangle)\\
&=\delta_{n,m}\frac 1{\zeta_N(\beta)}\sum_{\{s\in P: r\leq s\}}N(s)^{-\beta}\sum_{l=0}^{N_{r^{-1}s}-1}
\tau(\langle \31^{r^{-1}s}_l,\varphi_{r^{-1}s}(a)\31^{r^{-1}s}_l\rangle);
\end{align*}
comparing this last term with \eqref{omegatau-left} proves the claimed scaling identity.
\end{proof}

\begin{proof}[Proof of Theorem~\ref{thm:KMsbeta-from-tracetau}.] For every $\beta>\beta_c$, the
assignment $\tau\mapsto \omega_\tau\circ\Phi^\delta$ is continuous between compact Hausdorff spaces. It also respects convex linear combinations and weak$^*$-limits. If $\tau$ is a tracial state of
$A$, then $\omega_\tau$ is a trace on $\Ff$ by Theorem~\ref{thm:from-trace-A-to-trace-F}, which applies due to
Lemma~\ref{lemma:1} and Lemma~\ref{lemma:2}. Hence $\omega_\tau\circ \Phi^\delta$ is a KMS$_\beta$ state.
\end{proof}

\begin{prop}
Suppose that $A$ contains a proper isometry. Let $X$ be a compactly aligned product system
of finite type over $P$
of right Hilbert $A$--$A$-bimodules such that $\mathfrak{m}_{s,r}$
is bijective for all $s,r\in P$. Let $N:G\to (0,\infty)$ be a homomorphism such that $N(s)=N_s$ for all
$s\in P$, and let $\sigma^N$ be the corresponding dynamics on $\NT(X)$. Then there are no KMS$_\beta$
states of $(\NT(X),\sigma^N)$ for $\beta<1$.
\end{prop}

\begin{proof}
Suppose $\phi$ is a state of $\NT(X)$ which satisfies the KMS$_\beta$ condition for  some
$\beta>0$. Let $a\in A$ be a proper isometry. Then
$$
\phi(i_s(\varphi_s(a)\31^s_j)i_s(\varphi_s(a)\31^s_k)^*)
=N(s)^{-\beta}\phi(i_s(\31^s_k)^*i_e(a^*a)i_s(\31^s_j))=\delta_{j,k}N(s)^{-\beta}.
$$
Since $\sum_{j=0}^{N_s-1}i_s(\varphi_s(a)\31^s_j)i_s(\varphi_s(a)\31^s_j)^*$ is a projection, we have
 $$
 1= \phi(1)\geq \phi(\sum_{j=0}^{N_s-1}i_s(\varphi_s(a)\31^s_j)i_s(\varphi_s(a)\31^s_j)^*)=
\sum_{j=0}^{N_s-1}N(s)^{-\beta}=N_s^{1-\beta}.
 $$
Then necessarily $\beta\geq 1.$
\end{proof}

\begin{rmk}
\rm
We note that for a fixed tracial state $\tau$ of $A$, there exists a KMS$_\infty$ state $\omega_\infty$
of $\NT(X)$ obtained as the weak$^*$-limit of the KMS$_\beta$ states $\omega_\tau$ by letting
$\beta\to \infty$ in \eqref{def:state-from-induced}.
\end{rmk}


\section{Structure of the core $\Ff$ and reconstruction of KMS states}\label{section:core}

In this section, we start by identifying an action of $P$  by endomorphisms of the core $\Ff$ and a commutative
$C^*$-subalgebra $\Aa$ of $\Ff$. We then prove that the exact analogue of the reconstruction formula
for KMS$_\beta$ states of $\Tt(\NN\rtimes\NN^\times)$ from \cite[Lemma 10.1]{Lac-Rae2} is valid, with similar proofs, for $\NT(X)$ of a product system under mild assumptions on the semigroup $P$ and the function $\zeta_N$ defined in
(\ref{def:series}). Using this, we will prove surjectivity and injectivity of  the parametrization of KMS states given in Theorem~\ref{thm:KMsbeta-from-tracetau} for certain classes of product systems of finite type.

\begin{prop}\label{prop:endo-of-F}
Let $(G, P)$ be lattice ordered and $X$ a product system over $P$ of finite type such that the maps
$\mathfrak{m}_{s,r}$ are bijective for all $s,r\in P$. The assignment
\begin{equation}\label{eq:def-alpha-s}
\alpha_s(y)=\sum_{j=0}^{N_s-1} i_s(\31^s _j) yi_s(\31^s _j)^*
\end{equation}
for $s\in P$ and $y\in \Ff$ defines an action $\alpha$ of $P$ on $\Ff$ by injective endomorphisms.
\end{prop}

\begin{proof}
Clearly $\alpha_s:\Ff\to\Ff$ is well-defined, and $\alpha_s(y_1)\alpha_s(y_2)=\alpha_s(y_1y_2)$ for all
 $s\in P$ and  $y_1,y_2\in \Ff$ follows by Definition~\ref{def:finite-type}(3). If $\alpha_s(y)=0$ then
$0=i_s(\31^s_0)^*\alpha_s(y)i_s(\31^s_0)=y$, so each $\alpha_s$ is injective. Let $s,r\in P$. Then
\begin{align*}
\alpha_s\alpha_r(y)
&=\alpha_s\left(\sum_{j=0}^{N_s-1}i_r(\31^r_l)yi_r(\31^r_l)^*\right)\\
&=\sum_{j=0}^{N_s-1}\sum_{l=0}^{N_r-1}i_{sr}(\31^{sr}_{j\cdot l})y (i_{sr}(\31^{sr}_{j\cdot l})^*) \\
&=\sum_{k=0}^{N_{sr}-1} i_{sr}(\31^{sr}_{k})y (i_{sr}(\31^{sr}_{k})^*)=\alpha_{sr}(y),
\end{align*}
which proves that $\alpha$ is an action of $P$ by endomorphisms of $\Ff$.
\end{proof}

We note that in \cite{HuR} similar constructions in the case of a single Hilbert bimodule and at the level of relative
Cuntz-Pimsner algebras (modeling Exel crossed products) are obtained.

\begin{cor} The maps $\alpha_r^q:B_r\to B_q$ given by $\alpha_r^q:=\alpha_{r^{-1}q}$ for $r\leq q$
give rise to a direct limit $\varinjlim_{r\in P}(B_r, \alpha_r^q)_{r\leq q}$ with injective homomorphisms. The canonical embeddings
$\alpha^r$  of $B_r$ into $\varinjlim_{r\in P}B_r$ give rise to an increasing union such that
$$
\Ff=\overline{\bigcup_{r\in P} \alpha^r(B_r)}.
$$
\end{cor}

\begin{proof}
Definition~\ref{def:finite-type}(2) implies that $\alpha_s(B_r)\subseteq B_{rs}$ for all $r,s\in P$, so the maps
$\alpha_r^q$ are well-defined from $B_r$ to $B_q$ for all $r\leq q$. The fact that $\alpha$ is an action of $P$ implies that
$\alpha_r^s=\alpha_s^q\circ \alpha_r^q$ when $r\leq q\leq s$, so the maps are compatible and give therefore rise to a direct system, as claimed. The inclusion maps $B_r\hookrightarrow \Ff$ for $r\in P$ are compatible with
the connecting maps $\alpha_r^q$, and combine to give an injective homomorphism from
$\varinjlim_{r\in P}B_r$ into $\Ff$, which is also surjective.
\end{proof}

\begin{prop}\label{prop:product-alphas}
For all $s,r\in P$ we have $\alpha_s(1)\alpha_r(1)=\alpha_{s\vee r}(1)$ and, consequently,
 $\Aa:=\clsp\{\alpha_r(1): r\in P\}$ is a commutative $C^*$-subalgebra of $\Ff$.
\end{prop}

\begin{proof}
By formula (\ref{eq:def-alpha-s}), $\alpha_s(1)=\sum_{j=0}^{N_s-1}i_s(\31^s_j)i_s(\31^s_j)^*$ and $\alpha_r(1)=\sum_{l=0}^{N_r-1}i_r(\31^r_l)i_r(\31^r_l)^*$ for $s,r\in P$. It follows from
\eqref{Neals-precise-formula} that $i_s(\31^s_j)^*i_r(\31^r_l)$ is a sum of terms
indexed over $i=0, \dots ,N_{s\vee r}-1$ where non-zero terms occur when $i(s)=j$
and $i(r)=l$ simultaneously. Thus
$$
\alpha_s(1)\alpha_r(1)=\sum_{j=0}^{N_s-1}\sum_{l=0}^{N_r-1}i_{s\vee r}
(\31^{s\vee r}_{j\cdot j'})i_{s\vee r}(\31^{s\vee r}_{l'\cdot l})^*,
$$
where the elements $l'=0, \dots, N_{r^{-1}(s\vee r)}-1$ and $j'=0, \dots, N_{s^{-1}(s\vee r)}-1$ are
such that $l'\cdot l=j\cdot j'$ in $\{0, \dots ,N_{s\vee r}-1\}$. Therefore $\alpha_s(1)\alpha_r(1)=
\alpha_{s\vee r}(1)$, as claimed.
\end{proof}

We are now ready to prove a reconstruction formula for KMS states similar to the namesake
formula in \cite[Lemma 10.1]{Lac-Rae2}.
As assumed so far in this section, $(G, P)$ is lattice ordered and $X$ is a product system over $P$ of finite type such that the maps $\mathfrak{m}_{s,r}$ are bijective for all $s,r\in P$. Let $N:G\to (0, \infty)$ be an injective homomorphism such that $N(s)=N_s$ for all $s\in P$, and assume that ${\zeta}^0(\beta):=\sum_{s\in P}N_s^{-\beta}$ admits an infinite Euler product
\begin{equation}\label{eq:inf-prod}
{\zeta}^0(\beta)=\prod_{s\in P}(1-N_s^{-\beta})^{-1}
\end{equation}
in an interval $(\beta_0, \infty)$ for some $\beta_0>0$.
Let $\Lambda$ be the directed set consisting of the finite subsets $F\subset P$ ordered under inclusion.

\begin{lemma}\label{lemma:reconstruction}
Assume $P$ has no non-trivial minimal elements. Let $\phi$ be a KMS$_\beta$ state of $\NT(X)$ and
let $(H_\phi, \xi_\phi, \pi_\phi)$ be the corresponding GNS representation. Denote by $\tilde{\phi}$ the vector state
 on $B(H_\phi)$ which extends $\phi$. For a finite subset $F$ of $P$ let
 \begin{equation}\label{eq:infinite-product}
Q_F:=\prod_{s\in F}\prod_{j=0}^{N_s-1}(I-\pi_\phi(i_s(\31^s_j)i_s(\31^s_j)^*)).
\end{equation}
Then $\{Q_F\}_{F\in \Lambda}$ converges in the weak operator topology to a projection
$Q$ in $\pi_\phi(\NT(X))^{''}$ which satisfies
\begin{enumerate}
\item $\tilde{\phi}(Q)={\zeta}^0(\beta -1)^{-1}$;
\item if ${\zeta}^0(\beta -1)<\infty$, then
$$\phi_{Q}(T):={\zeta}^0(\beta -1)\tilde{\phi}(Q \pi_\phi(T)Q), \;\; T\in \NT(X), $$
defines a state of $\NT(X)$;
\item if ${\zeta}^0(\beta -1)<\infty$, then
$$\phi(T)=\sum_{s\in P}\frac {N_s^{-\beta}}{{\zeta}^0(\beta -1)}\sum_{j=0}^{N_s-1}\phi_{Q}(i_s(\31^s_j)^*Ti_s(\31^s_j))$$
for all $T\in \NT(X)$.
\end{enumerate}
\end{lemma}

\begin{proof}
The proof runs basically as in \cite{Lac-Rae2}, mutatis mutandis. We first
note that  Definition~\ref{def:finite-type}(3) implies that $\prod_{j=0}^{N_s-1}(1-i_s(\31^s_j)i_s(\31^s_j)^*)=
1-\alpha_s(1)$, and therefore $Q_F= \prod_{s\in F} (I-\pi_\phi(\alpha_s(1)))$ for all $F\in \Lambda$.
Let $s\wedge r=e$.
We claim that $\phi\bigl((1-\alpha_s(1))(1-\alpha_r(1))\bigr)=\phi(1-\alpha_s(1))\phi(1-\alpha_r(1))$. Since
$s\wedge r=e$, Proposition~\ref{prop:product-alphas} implies that
$\alpha_s(1)\alpha_r(1)=\alpha_{sr}(1)$. Now the KMS condition implies that
\begin{align*}
\phi\bigl((1-\alpha_s(1))(1-\alpha_r(1))\bigr)
&=1-\phi(\alpha_s(1))-\phi(\alpha_r(1))+\phi(\alpha_{sr}(1))\\
&=1-N_s^{1-\beta} -N_r^{1-\beta}+N_{sr}^{1-\beta}\\
&=(1-N_s^{1-\beta})(1-N_r^{1-\beta})\\
&=\phi(1-\alpha_s(1))\phi(1-\alpha_r(1)).
\end{align*}
Hence $\phi(Q_F)=\prod_{s\in F}(1-N_s^{1-\beta})$, and claims (1) and (2) follow exactly as in the proof of
the analogous statements in \cite{Lac-Rae2}.

For (3), we need to know that $\{\pi_\phi(i_s(\31^s_j))Q\pi_\phi(i_s(\31^s_j))^*:j=0, \dots, N_s-1, s\in P\}$
is a family of mutually orthogonal projections. Clearly these elements are projections. Let $s,r\in P$.
If $s=r$ then two different projections in the family correspond to $k\neq l\in\{0, \dots, N_s-1\}$,
and they are orthogonal because $\langle \31^s_k, \31^s_l\rangle=0$. If $s\neq r$ then either
$r'=s^{-1}(s\vee r)$ or $s'=r^{-1}(s\vee r)$ is different from the identity $e$ in $P$. Thus there is a $q\in P$
such that either $q\leq r'$ or $q\leq s'$. It follows from  \eqref{Neals-precise-formula}
that $i_s(\31^s_k)^*i_r(\31^r_l)$ is a sum over $i=0, \dots ,N_{s\vee r}-1$ of terms
$$
i_{r'}(\31^{r'}_{i(r')})i_{s'}(\31^{s'}_{i(s')})^*,
$$
where $i(r')=0, \dots, N_{r'}-1$, $i(s')=0,\dots ,N_{s'}-1$, $i(s)=k$ and $i(r)=l$. If for example there is a $q\in P$ with
$q\leq r'$, then in the product
$Q \pi_\phi(\31^s_k)^*\pi_\phi(\31^r_l) Q$ there will be a middle term of the form
$$
\prod_{p\in P}(I-\pi_\phi(\alpha_p(1)))\pi_\phi(i_q(\31^q_h))\pi_\phi(i_{q^{-1}r'}(\31^{q^{-1}r'}_g))
$$
for some $g=0,\dots, N_q-1$ and $h=0,\dots, N_{q^{-1}r'}-1$. Since $(1-\alpha_q(1))i_q(\31^q_h)=0$,
orthogonality follows. The argument is similar in the case $q\leq s'$.

In the next step we carry on as in \cite{Lac-Rae2} to conclude that
$$
\tilde{\phi}\bigg(\sum_{s\in P,\;0\leq j\leq N_s-1}\pi_\phi(i_s(\31^s_j))Q\pi_\phi(i_s(\31^s_j))^*\bigg)=1,
$$
except that we need to verify
\begin{equation}\label{forgotten}
\tilde\phi(\pi_\phi(\31^s_k)Q\pi_\phi(\31^s_k)^*)=N_s^{-\beta}\tilde\phi(Q)
\end{equation}
for every $k=0, \dots ,N_s-1$.
 A similar equality seems to be required in the first displayed formula from \cite[page 681]{Lac-Rae2}. As we shall
now see, \eqref{forgotten} follows from our scaling identity \eqref{scaling}. Indeed, by the definition of
$Q$ we can write
\begin{align*}
\tilde\phi(\pi_\phi(\31^s_k)Q\pi_\phi(\31^s_k)^*)
&=\bigl(\pi_\phi (\31^s_k)Q\pi_\phi(\31^s_k)^*\xi_\phi \mid \xi_\phi  \bigr)_{H_\phi}\\
&=\lim_F \bigl( Q_F \pi_\phi(\31^s_k)^*\xi_\phi \mid \pi_\phi(\31^s_k)^*\xi_\phi\bigr)_{H_\phi}\\
&=\lim_F \bigl( \prod_{r\in F} \pi_\phi((1-\alpha_r(1))i_s(\31^s_k)^*)\xi_\phi \mid  \pi_\phi(\31^s_k)^*\xi_\phi\bigr)_{H_\phi}\\
&=\lim_F \phi(i_s(\31^s_k) \prod_{r\in F}(1-\alpha_r(1)) i_s(\31^s_k)^*)\\
&=\lim_F N_s^{-\beta}\phi(\prod_{r\in F}(1-\alpha_r(1)) \;\;\;
\text{ by }\eqref{scaling};\\
\end{align*}
the last term is precisely the right-hand side of \eqref{forgotten}, as needed. Thus it follows that
$$
\phi(T)=\sum_{s=r, j=l} N_s^{-\beta}\tilde{\phi}(Q \pi_\phi(i_s(\31^s_j))^* \pi_\phi(T) \pi_\phi(i_r(\31^r_l))Q),
$$
for $T$ in $\NT(X)$, giving (3).
\end{proof}


\begin{lemma}\label{lemma:phiQ-trace-onA}
Let $(G, P)$ be lattice ordered and let $X$ be an
associative product system of finite type over $P$ such that $\mathfrak{m}_{s,r}$ are bijective and respect co-prime pairs, for all $s,r\in P$. Let  $\zeta_0$  be given by \eqref{eq:inf-prod} and assume that $P$ has no non-trivial minimal elements. Given a KMS$_\beta$ state $\phi$ of $\NT(X)$, where $\beta>\beta_0$, let $\phi_Q$ be its associated state constructed in  Lemma~\ref{lemma:reconstruction}.
If the restriction $\phi_Q\circ i_e$ of $\phi_Q$ to $A$ is a tracial state, then $\phi$  coincides with the KMS$_\beta$ state $\omega_{\phi_Q\circ i_e}\circ \Phi^\delta$ given by \eqref{def:state-from-induced}.
\end{lemma}

\begin{proof}
Assume $\phi$ is a KMS$_\beta$ state of $\NT(X)$ such that $\phi_Q\circ i_e$ is a tracial state on $A$. Since both $\omega_{\phi_Q\circ i_e}$ and $\phi$ are supported on $\Ff$ by Proposition~\ref{prop:KMS-state-on-F}, it suffices to prove that
$\omega_{\phi_Q\circ i_e}(y)=\phi(y)$ for all $y\in \Ff$.  Since $\zeta_0$ is given by \eqref{eq:inf-prod} (and in particular
$N_s=N(s)$), we have $\zeta_N(\beta)=\zeta^0(\beta-1)$, and the series is convergent in the interval $(\beta_0,\infty)$. Let $y=i^{(r)}(\theta_{\xi,\eta})$ in $\Ff$. We adopt the convention that
$i_r^s:\Ll(X_r)\to \Ll(X_s)$ is the zero map when $s\notin rP$.  Equations
\eqref{eq:sum-of-vector-states} and \eqref{eq:Ind-pi-tau} imply that
\begin{align}
\omega_{\phi_Q\circ i_e}(y)
&=\sum_{s\in P}\frac{N_s^{-\beta}}{\zeta^0(\beta-1)}\sum_{j=0}^{N_s-1}{\phi_Q\circ i_e}(\langle \31^s_j, i_r^s({\theta_{\xi,\eta}})\31^s_j\rangle)\notag\\
&=\sum_{s\in P}\frac{N_s^{-\beta}}{\zeta^0(\beta-1)}\sum_{j=0}^{N_s-1} \phi_Q\bigl(i_s(\31^s_j)^*i_s(i_r^s(\theta_{\xi,\eta})\31^s_j)\bigr).\label{get-phi-from-omega1}
\end{align}
We look separately at elements of the form $i_s(i_r^s(\theta_{\xi,\eta})\31^s_j)$. For each $j$ there is a unique decomposition
$\31^s_j=\31^r_{k}\otimes \31^{r^{-1}s}_l$ with $k\in \{0,\dots, N_r-1\}$ and $l\in \{0,\dots,N_{r^{-1}s}-1\}$. Since $i$
is a representation of the product system, we have that
\begin{align*}
i_s(i_r^s(\theta_{\xi,\eta})\31^s_j)
&=i_s\bigl(F^{r,r^{-1}s}(\xi\otimes_A \varphi_{r^{-1}s}(\langle \eta,\31^r_k\rangle)\31^{r^{-1}s}_l \bigr)\\
&=i_r(\xi)i_{r^{-1}s}(\varphi_{r^{-1}s}(\langle \eta,\31^r_k\rangle)\31^{r^{-1}s}_l)\\
&=i_r(\xi)i_e(\langle \eta,\31^r_k\rangle)i_{r^{-1}s}(\31^{r^{-1}s}_l)\\
&=i_r(\xi)i_r(\eta)^*i_r(\31^r_k)i_{r^{-1}s}(\31^{r^{-1}s}_l)\\
&=yi_s(\31^s_j).
\end{align*}
Inserting this in \eqref{get-phi-from-omega1} implies that
$$
\omega_{\phi_Q\circ i_e}(y)=\sum_{s\in P}\frac{N_s^{-\beta}}{\zeta^0(\beta-1)}\sum_{j=0}^{N_s-1} \phi_Q\bigl(i_s(\31^s_j)^*)yi_s(\31^s_j),
$$
and an application of Lemma~\ref{lemma:reconstruction}(3) therefore gives that $\omega_{\phi_Q\circ i_e}(y)=\phi(y)$, as claimed.
\end{proof}

The following corollary describes conditions guaranteeing surjectivity of the parametrization of KMS states.

\begin{cor}\label{cor:surjectivity-of-parametrization}
Under the hypotheses of Lemma~\ref{lemma:phiQ-trace-onA}, assume moreover
that for every KMS$_\beta$ state $\phi$ of $\NT(X)$, where $\beta>\beta_0$, the restriction
$\phi_Q\circ i_e$ to $A$ is a tracial state. Then the
parametrization of KMS$_\beta$ states for $\beta>\beta_0$ from Theorem~\ref{thm:KMsbeta-from-tracetau} is surjective.

Furthermore, the condition that $\phi_Q\circ i_e$ is a tracial state on $A$ for every KMS$_\beta$ state
$\phi$ of $\NT(X)$ holds in particular in the following cases:
\begin{enumerate}
\item $A$ is commutative.
\item $i_e(A)$ commutes with $\{\alpha_s(1): s\in P\}$.
\end{enumerate}
\end{cor}

\begin{proof}
The first part follows from Lemma~\ref{lemma:phiQ-trace-onA}. The second part is immediate from the definitions
of $Q$ and $\phi_Q$ in Lemma~\ref{lemma:reconstruction}, because $Q_F= \prod_{s\in F} (I-\pi_\phi(\alpha_s(1)))$.
\end{proof}

Now we address the problem of injectivity of the map from tracial states of $A$ to KMS$_\beta$ states
on $\NT(X)$ constructed in Theorem~\ref{thm:KMsbeta-from-tracetau}. We have a positive
answer under additional assumptions on the product system. First recall that a finite subset $F$
of $P$ is $\vee$-closed if $r\vee s\in F$ for all $r,s\in F$. A finite subset $F\subseteq P$
will be called \emph{strictly $\vee$-closed} if $r\vee s\in F$ for some $s\in F$ and
$r\in P$ implies that $r\in F$. Also, we denote by $\Lambda^\vee$  the directed set consisting
of $\vee$-closed finite subsets of $P$ ordered under inclusion.

\begin{theorem}\label{thm:inj-of-parametrization}
Let $(G, P)$ be lattice ordered and let $X$ be an
associative product system of finite type over $P$ such that $\mathfrak{m}_{s,r}$ are bijective and respect co-prime pairs, for all $s,r\in P$. Let  $\zeta_0$  be given by \eqref{eq:inf-prod} and assume that $P$ has no non-trivial minimal elements. For a KMS$_\beta$ state $\phi$ of $\NT(X)$, where $\beta>\beta_0$, let $\phi_Q$ be its associated state constructed in  Lemma~\ref{lemma:reconstruction}.
Assume that $\phi_Q\circ i_e$ is a tracial state on $A$ for every KMS$_\beta$ state with $\beta>\beta_0$.

If there is a subset $A_0$ of $A$ with dense linear span and such that for every
$a\in A_0$ there is a finite strictly $\vee$-closed
subset $F_a$ of $P$ with the property that
\begin{equation}\label{eq:zero-outside-r}
\sum_{j=0}^{N_s-1}\langle \31^s_j, \varphi_s(a)\31^s_j\rangle=0 \text{ for all }s\notin (F_a\cup \{e\}),
\end{equation}
then the parametrization of KMS$_\beta$ states from Theorem~\ref{thm:KMsbeta-from-tracetau} is injective.
\end{theorem}

\noindent
{\em Proof.}
Let $\tau$ be a tracial state on $A$. Let $\beta>\beta_0$ and let $\phi:=\omega_\tau\circ\Phi^\delta$ be the associated  KMS$_\beta$ state, as in Theorem~\ref{thm:KMsbeta-from-tracetau}. It suffices to prove that
\begin{equation}\label{eq:tau-back}
\phi_Q\circ i_e(a)=\tau(a) \text{ for all }a\in A_0.
\end{equation}

Fix $a\in A_0$. We write the corresponding strictly $\vee$-closed finite subset $F_a$ of $P$ in the form
\begin{equation}\label{eq:Fa}
F_a:=F_a'\cup\bigcup_{J\subseteq F_a', \vert J\vert\geq 2}\{p_J\},
\end{equation}
where $F_a'=\{p_1,\dots ,p_n: p_i\not\leq p_j \text{ for } i\neq j, 1\leq i,j\leq n\}$, and
$p_J:=\vee_{p\in J}p$ for every $J\subseteq F_a'$ with $\vert J\vert \geq 2$. In particular, $e\notin F_a.$

By the definition of $\phi_Q$ and the fact that $\phi$ restricted to $\Ff$ is the trace $\omega_\tau$,
we have
$$
(\phi_Q\circ i_e)(a)=\zeta^0(\beta-1)\varinjlim_{F,F'\in \Lambda}\omega_\tau\bigl(i_e(a)\prod_{r\in F}(1-\alpha_r(1))\prod_{q\in F'}(1-\alpha_q(1))\bigr).
$$
By Proposition~\ref{prop:product-alphas}, any product $(1-\alpha_r(1))(1-\alpha_q(1))$ can
again be written in the  form $1- (\alpha_r(1)+\alpha_q(1) +
\alpha_{r\vee q}(1))$, so replacing $\Lambda$ with $\Lambda^{\vee}$ we may assume that
$$(\phi_Q\circ i_e)(a)=\zeta^0(\beta-1)\varinjlim_{F\in \Lambda^{\vee}}\omega_\tau(i_e(a)\prod_{r\in F}(1-\alpha_r(1))).$$
Now for every $r\in P\setminus\{e\}$ we have
\begin{align}
\omega_\tau(i_e(a)\alpha_r(1))
&=\sum_{k=0}^{N_r-1}\omega_\tau(\theta_{\varphi_r(a)\31^r_k, \31^r_k})\notag\\
&= \sum_{k=0}^{N_r-1} \sum_{s\geq r}\frac {N_s^{-\beta}}{\zeta^0(\beta-1)}\sum_{j=0, j=k\cdot j'}^{N_s-1}
\tau(\langle (\langle\varphi_{r^{-1}s}(a)\31^r_k, \31^r_k\rangle\31^{r^{-1}s}_{j'}, \31^{r^{-1}s}_{j'}\rangle)\notag\\
&=\sum_{s\geq r}\frac {N_s^{-\beta}}{\zeta^0(\beta-1)}\sum_{k=0}^{N_r-1}\sum_{j=0, j=k\cdot j'}^{N_s-1}
\tau(\langle \31^s_{k\cdot j'}, \varphi_s(a)\31^s_{k\cdot j'}\rangle)\notag\\
&=\sum_{s\geq r}\frac {N_s^{-\beta}}{\zeta^0(\beta-1)} \sum_{j=0}^{N_s-1}
\tau(\langle \31^s_{j}, \varphi_s(a)\31^s_{j}\rangle).\label{eq:omega-tau-a-r}
\end{align}
Since $F_a$ is strictly $\vee$-closed, it follows that $s\notin F_a$ whenever $r\notin F_a$ and $s\geq r$.  Hence \eqref{eq:omega-tau-a-r} and the assumption \eqref{eq:zero-outside-r} imply that $\omega_\tau(i_e(a)\alpha_r(1))=0$ when $r\notin F_a$. The Cauchy-Schwartz inequality implies therefore that for every $F\in \Lambda^\vee$ with $e\notin F$,
\begin{equation}\label{eq:omega-tau-zero}
\omega_\tau\bigg(i_e(a)\prod_{\{r\in F, F\setminus F_a\neq \emptyset\}}\alpha_r(1)\bigg)=0.
\end{equation}
Using induction, it is straightforward to see that for every finite subset $F$ of $P$ we have
$$
\prod_{q\in F}(1-\alpha_q(1))=1+\sum_{\{J\subseteq F, \,J\text{ finite}\}}(-1)^{\vert J\vert }\prod_{p\in J}\alpha_{p}(1).
$$
Combining this with Proposition~\ref{prop:product-alphas} we thus see that
$$
\omega_\tau(i_e(a) \prod_{q\in F}(1-\alpha_q(1)))=\omega_\tau(i_e(a))+\sum_{\{J\subseteq F, \,J\text{ finite}\}}(-1)^{\vert J\vert }\omega_\tau(i_e(a)\alpha_{p_J}(1)).
$$
 for every $F\in \Lambda^\vee$, $e\notin F$.
By \eqref{eq:omega-tau-zero}, only the terms corresponding to $J\subseteq F_a$ can give a non-zero
contribution in
$\omega_\tau(i_e(a) \prod_{q\in F}(1-\alpha_q(1)))$ provided that $F\cap F_a\neq \emptyset$. Since the family $\Lambda^a=\{F\in \Lambda^\vee: F_a\subseteq F\}$ is cofinal in $\Lambda^\vee$, we have
\begin{align*}
(\phi_Q\circ i_e)(a)
&=\zeta^0(\beta-1)\varinjlim_{\Lambda^a}\omega_\tau(i_e(a) \prod_{q\in F}(1-\alpha_q(1)))\\
&=\zeta^0(\beta-1)\bigl(\omega_\tau(i_e(a))+\sum_{\{J\subseteq F_a, \,J\text{ finite}\}}(-1)^{\vert J\vert }\omega_\tau(i_e(a)\alpha_{p_J}(1))\bigr)\\
&=\zeta^0(\beta-1)\omega_\tau(i_e(a) \prod_{q\in F_a}(1-\alpha_q(1))).\\
\end{align*}
Since every element $1-\alpha_q(1)$ with $q\in F_a'$
is dominated by any $1-\alpha_{p_J}(1)$ whenever $J\subset F_a'$ with $q\in J$, we have
$$\prod_{q\in F_a}(1-\alpha_q(1))=\prod_{q\in F_a'}(1-\alpha_q(1)).$$
Thus to prove \eqref{eq:tau-back} it suffices to show that
$$
\omega_\tau(i_e(a))+\sum_{\{J\subseteq F_a', \,J\text{ finite}\}}(-1)^{\vert J\vert }\omega_\tau(i_e(a)\alpha_{p_J}(1))=\frac 1{\zeta^0(\beta -1)}\tau(a).
$$
 Equivalently, using the hypothesis that $a\in A_0$ in the definition of $\omega_\tau(i_e(a))$, we must show that
 \begin{equation}\label{eq:big-sums-give-zero}
 \sum_{s\in F_a} \frac {N_s^{-\beta}}{\zeta^0(\beta-1)}\sum_{j=0}^{N_s-1}
\tau(\langle \31^s_{j}, \varphi_s(a)\31^s_{j}\rangle) + \sum_{\{J\subseteq F_a', \,J\text{ finite}\}}(-1)^{\vert J\vert }\omega_\tau(i_e(a)\alpha_{p_J}(1))=0.
 \end{equation}
To prove this equality we need a lemma.

\begin{lemma}\label{lem:big-sums-give-zero}
Assume the hypotheses of Theorem~\ref{thm:inj-of-parametrization}. Let $\tau$ be a tracial state on $A$, $a\in A_0$ and  $F_a$ be as in \eqref{eq:Fa}. For every $s\in P$ denote $\lambda_s=N_s^{-\beta}\sum_{j=0}^{N_s-1}\tau(\langle \31^s_{j}, \varphi_s(a)\31^s_{j}\rangle)$
and let $n=\vert F_a'\vert$. Then
\begin{equation}\label{eq:why-big-sums-give-zero}
\sum_{k=1}^n\sum_{\{J\subseteq F_a, \vert J\vert=k \}}\lambda_{p_J} + \sum_{k=1}^n\sum_{\{J\subseteq  F_a', \vert J\vert=k \}} (-1)^k\sum_{\{s\in F_a, p_J\leq s\}}\lambda_s=0.
\end{equation}
\end{lemma}

\begin{proof}
We prove \eqref{eq:why-big-sums-give-zero} by induction on $n$. The statement is immediate when $F_a'=\{p_1\}$. If $F_a'=\{p_1, p_2\}$, the first sum in  \eqref{eq:why-big-sums-give-zero} is equal to $\lambda_{p_1} + \lambda_{p_2} + \lambda_{p_1\vee p_2}$, while the second is $-(\lambda_{p_1} + \lambda_{p_1\vee p_2}+\lambda_{p_2}+\lambda_{p_1\vee p_2} ) + \lambda_{p_1\vee p_2}$, so they add up to zero, as claimed.

Assume \eqref{eq:why-big-sums-give-zero} is true when  $F_a'=\{p_1, \dots, p_n\}$ and let $q\in P\setminus \{e\}$ such that
$p_j\not\leq q$ or $q\not\leq p_j$ for all $j=1, \dots ,n$. In order to prove that \eqref{eq:why-big-sums-give-zero} is valid for
$F_a'\cup \{q\}$ it suffices to show that the passage from $F_a'$ to $F_a'\cup \{q\}$ gives zero contribution in the left-hand side. Now, the terms in the left-hand side of \eqref{eq:why-big-sums-give-zero} at $n+1$ which depend on $q$ give the contribution
$$
\sum_{k=1}^{n+1}\sum_{\{J\subseteq F_a', \vert J\vert =k-1\}} \lambda_{p_J\vee q} +
\sum_{k=1}^{n+1} \sum_{\{J\subseteq F_a', \vert J\vert =k-1\}} (-1)^k\sum_{\{J'\subseteq F_a', 0\leq \vert J'\vert\leq k\}} \lambda_{q\vee p_{J\cup J'}}.
$$
Regrouping the terms, this becomes
$$
\sum_{k=1}^{n+1}\sum_{\{J\subseteq F_a', \vert J\vert =k-1\}} \sum_{l=0}^k(-1)^l {k \choose l}\lambda_{p_J\vee q},
$$
which is zero because $\sum_{l=0}^k(-1)^l {k \choose l}=0$. This finishes the proof of the lemma.
\end{proof}

\noindent
{\em End of proof of Theorem~\ref{thm:inj-of-parametrization}.} To see that \eqref{eq:big-sums-give-zero}
follows from Lemma~\ref{lem:big-sums-give-zero}, just note that \eqref{eq:omega-tau-a-r} implies that
$$
\omega_\tau(i_e(a)\alpha_{p_J}(1))=\frac 1{\zeta^0(\beta-1)} \sum_{\{s\in F_a, p_J\leq s\}}\lambda_s.
$$
\hfill$\Box$

\begin{rmk}\label{comparisonwithLR}
\rm{
We claim that condition (2) from Corollary~\ref{cor:surjectivity-of-parametrization} is satisfied for $\Tt(\NN\rtimes\NN^\times)$, viewed as the Nica-Toeplitz algebra $\NT(M_K)$ of a product system $M_K$ over $\NN^\times$ of Hilbert $C^*$-correspondences over $\Tt$, cf. \cite[Theorem 6.6]{BanHLR}. In particular, we recover the surjectivity claim from \cite[Proposition 10.5]{Lar-Rae}.

We let $S$ denote  the generating isometry in $\Tt$ and $s=i_1(S)$ be its image in $\NT(M_K)$.  Let $v_r$ be
the isometry in $\NT(M_K)$ corresponding to the basis element $\31^r_0$ in $(M_K)_r$ for every $r\in \NN^\times$. Recall from \cite{Lac-Rae2} that $v_rs=s^rv_r$ (relation (T1)) and $s^*v_r=s^{r-1}v_rs^*$ (relation (T4)). Hence
$$
ss^*v_r=s^rv_rs^*=v_rss^*
$$
and by taking adjoints also $(ss^*v_r^*)^*=v_rss^*=ss^* v_r=(v_r^*ss^*)^*$. Thus $ss^*$ commutes with $v_r$ and $v_r^*$ for all
$r\in \NN^\times$.

We claim that  the image $i_1(\Tt)$ of $\Tt$ in $\NT(M_K)$ commutes with the projection $\alpha_r(1)$ for all $r\in \NN^\times$. To prove this, notice that $N_r=r$  and $i_r(\31^r_j)=s^jv_r$ for each $r\in \NN^\times$ and all $j=0,\dots, r-1$. The claim  amounts to proving that
\begin{equation}\label{eq:commutation-for-Taff}
s\bigl(\sum_{j=0}^{r-1}s^jv_rv_r^*s^{*j}\bigr)=\bigl(\sum_{j=0}^{r-1}s^jv_rv_r^*s^{*j}\bigr)s.
\end{equation}
The left-hand side of \eqref{eq:commutation-for-Taff} unfolds as follows
$$
s\bigl(\sum_{j=0}^{r-1}s^jv_rv_r^*s^{*j}\bigr)
=sv_rv_r^*+s^2v_rv_r^*s^*+s^3v_rv_r^*s^{*2}+\cdots
+s^{r-1}v_rv_r^*s^{*(r-2)}+s^rv_rv_r^*s^{*(r-1)}.
$$
Since $s^*s=1$, the right-hand side of \eqref{eq:commutation-for-Taff} becomes
$$
\bigl(\sum_{j=0}^{r-1}s^jv_rv_r^*s^{*j}\bigr)s
=v_rv_r^*s+sv_rv_r^*+s^2v_rv_r^*s^*+s^3v_rv_r^*s^{*2}+\cdots + s^{r-1}v_rv_r^*s^{*(r-2)}.
$$
Comparing the displayed sums, \eqref{eq:commutation-for-Taff} follows if we show that $v_rv_r^*s=s^rv_rv_r^*s^{*(r-1)}$. Using (T1), (T4), and the
fact that $ss^*v_r^*=v_r^*ss^*$, this last equality follows from the calculations
\begin{align*}
s^rv_rv_r^*s^{*(r-1)}
&=v_rs (s^{r-1}v_r)^*=v_rs(s^*v_rs)^*\\
&=v_rss^*v_r^*s=v_rv_r^*ss^*s=v_rv_r^*s.
\end{align*}}

Next we show that the condition \eqref{eq:zero-outside-r} is satisfied in this example, thereby also recovering the injectivity claim from  \cite[Proposition 10.5]{Lar-Rae}. The coefficient algebra $\Tt$ is the closed linear span of monomials of the form $S^mS^{*n}$ 
for $m,n\in \NN$. For $a=S^mS^{*n}$ let $F_a'$ be the set of distinct primes dividing $\vert m-n\vert$. We will show that 
every summand in \eqref{eq:zero-outside-r} vanishes when $r\notin F_a'$. Fix therefore $r$ not dividing $\vert m-n\vert$, $r\neq 1$, and let $0\leq j<r$. Let $V_r$ be the isometry defined in \cite[\S6]{BanHLR}. Then  $$
\langle\31^r_j,\varphi_r(a)\31^r_j\rangle=V_r^*(S^{*j}S^{m}S^{*n}S^j)V_r.
 $$
 Since $S$ is an isometry, the product $S^{*j}S^{m}S^{*n}S^j$ is of the form $S^kS^{*l}$ for $k,l\in \NN$, where $r$ does not divide $\vert k-l\vert$. Suppose  $k\leq l$. Then $V_r^*S^kS^{*l}V_r=V_r^*S^{k}S^{*k}S^{* (l-k)}V_r$. Using relations (T1) and (T4), one can verify using induction that
 $$
 V_r^*S^{k}S^{*k}=SS^*V_r^*
 $$
 for $r\geq k$. Hence the term $V_r^*S^{k}S^{*k}S^{*(l-k)}V_r$ contains a factor $V_r^*S^{*(l-k)}V_r$. But this is zero by relation (T5) in $\Tt(\NN\rtimes\NN^\times)$. The case $k>l$ is similar.
\end{rmk}

\begin{example}\label{brla}
\rm{We borrow from \cite{Bro-Lar} a class of examples of product systems of finite type.
Let $\Gamma$ be a discrete abelian group and $(G, P)$  a lattice ordered group.
Suppose $\alpha$ is an action of $P$ by endomorphisms
of $\Gamma$ such that $\alpha_p$ is injective for all $p\in P$ and the following are satisfied:
(i) $[\Gamma: \alpha_p(\Gamma)]<\infty$
for all $p\in P$; (ii) $\alpha_{p\vee q}(\Gamma)=\alpha_p(\Gamma)\cap \alpha_q(\Gamma)$
for all $p,q\in P$. We let $\hat\alpha$ be the action
of $P$ on $C^*(\Gamma)$ given by $\hat\alpha_p(\delta_\gamma)=\delta_{\alpha_p(\gamma)}$
for $p\in P$ and $\gamma\in \Gamma$. The map on $C^*(\Gamma)$ given by $L_p(\delta_\gamma)=\delta_{\alpha_p^{-1}(\gamma)}$ if $\gamma\in \alpha_p(\Gamma)$ and
$L_p(\delta_\gamma)=0$ if $\gamma\not\in \alpha_p(\Gamma)$ is a transfer operator for $\alpha_p$
for every $p$. Then there is a product system
$X$ over $P$ with the module $X_p$ for $p\in P$ equal as vector space with $C^*(\Gamma)$ and
with actions $f\cdot\xi\cdot g=f\xi\hat\alpha_p(g)$ and inner product $\langle \xi,\eta\rangle=
L_p(\xi^*\eta)$ for $f,\xi,\eta,g\in C^*(\Gamma)$. For each $p\in P$ let $N_p=[\Gamma:\alpha_p(\Gamma)]$
and let $\{\gamma^p_0,\dots,\gamma^p_{N_p-1}\}$ be a set of coset representatives for $\Gamma/\alpha_p(\Gamma)$. If $\31^p_j$ denotes the image of $\gamma^p_j$ in $X_p$,  then
$\{\31^p_j: 0\leq j\leq N_p-1\}$ is an orthonormal basis for $X_p$. If $N_pN_q=N_{pq}$ for all $p,q\in P$,
the product system $X$ is of finite type.

A concrete example of
this setup arises from taking $(G,P)=(\QQ^*_+,\NN^\times)$, $\Gamma=\ZZ^d$ for $d\geq 1$. The resulting product system $X$ over $P$ is of finite type, with $N_r=r$ for all $r\in \NN^\times$. The maps $\mathfrak{m}_{p,q}$ are bijective and respect co-prime pairs. Note that $\zeta^0(\beta)=\sum_{r\in \NN^\times}r^{-(\beta-1)}$ is convergent in the interval $(2,\infty)$. Thus Theorem~\ref{thm:KMsbeta-from-tracetau} applies, and gives KMS$_\beta$ states associated to probability measures on $\TT^d$ for every  $\beta>2$. Since the underlying $C^*$-algebra of the product system is $A=C(\TT^d)$, Corollary~\ref{cor:surjectivity-of-parametrization} applies as well.
}

\end{example}

\begin{rmk}\label{rmk:KMS-smallbeta}
\rm
Note that any KMS$_\beta$ state $\omega_\tau$ given by \eqref{def:state-from-induced}
at $\beta>\beta_c$ satisfies
\begin{equation}\label{eq:smallbeta}
\omega_\tau(i_r(\31^r_n)i_r(\31^r_m)^*)=
\begin{cases}N(r)^{-\beta}&\text{ if }n=m\\
0&\text{ otherwise. }\\
\end{cases}
\end{equation}
Hence $\omega_\tau(\alpha_r(1))=N_rN(r)^{-\beta}$ for all $r$. If $N(r)=N_r$ for all $r\in P$, this condition is $\omega_\tau(\alpha_r(1))=N(r)^{1-\beta}$. Also, the restriction of $\omega_\tau$ to $\Aa$
(see Proposition \ref{prop:product-alphas}) is independent of $\tau$. One can therefore
ask whether  KMS$_\beta$ states can be constructed by other methods, and possibly for a larger
range of $\beta$'s, by starting from states of $\Aa$ or a subalgebra hereof.
It is known that a KMS$_\beta$ state at every $\beta\geq 1$ exists in the case of the product system from  Example~\ref{ex-affine-toeplitz},  as shown in \cite[Proposition 9.1]{Lac-Rae2}. This
 state is supported  on a commutative $C^*$-subalgebra of $\Ff$, and we expect that
similar considerations could work  more generally.

 One would like to  apply \cite[Theorem 4.1]{Lac-Nes2} to the system $(\Ff\rtimes_\alpha P, \sigma)$,
where the dynamics $\sigma$ is trivial on
the image of $\Ff$ in $\Ff\rtimes_\alpha P$ and scales the implementing isometries
$v_s\in \Ff\rtimes_\alpha P$ by $N(s)^{it}$ for $s\in P$, $t\in\RR$. Then for every
$\beta\in \RR$, KMS$_\beta$ states on $\Ff\rtimes_\alpha P$ would be determined
by tracial states $\tau$ on $\Ff$ which satisfy the scaling condition
$\tau\circ \alpha_s=N(s)^{-\beta}\tau$ for every
$s\in P$. Note that for a tracial state $\tau$ on $\Ff$ to satisfy the scaling condition
we must have $\tau(\alpha_s(1)):=N(s)^{-\beta}$ for every $s\in P$. However, this last
equality does not match with $\omega_\tau(\alpha_s(1))=N_sN(s)^{-\beta}$, so states above
$\beta_c$ and states below $\beta_c$ would live on different subalgebras.

 \end{rmk}

\end{document}